\documentclass[makeidx]{amsart}
\usepackage[active]{srcltx}
\usepackage{atveryend}
\makeatletter
\let\origcitation\citation
\AtEndDocument{\def\mycites{}%
  \def\citation#1{\g@addto@macro\mycites{#1^^J}\origcitation{#1}}}
\AtVeryEndDocument{\newwrite\citeout\immediate\openout\citeout=\jobname.cit
  \immediate\write\citeout{\mycites}\immediate\closeout\citeout}
\makeatother
\usepackage{enumerate}
\usepackage{cancel}
\usepackage{mathdots}
\usepackage{epsfig}
\usepackage{hyperref}
\usepackage{amsmath}
\usepackage{amssymb}
\usepackage{graphics} 
\usepackage{yfonts}
\usepackage[usenames,dvipsnames]{color}
\def\phi{\varphi}
 \def\bflam{{\boldsymbol \lambda}}

\def\CC{\mathbb{C}}
 \def\RR{\mathbb{R}}
 \def\NN{\mathbb{N}}
\def\ZZ{\mathbb{Z}}

\def\lap{\mathcal{L}}

\def\bfk{\mathbf{k}}

\newtheorem{Proposition}{Proposition}

  \newtheorem{Corollary}[Proposition]{Corollary}
  \newtheorem{Lemma}[Proposition]{Lemma}
  
  \newtheorem{Theorem}[Proposition]{Theorem}
 
 \newtheorem{Definition}[Proposition]{Definition}
 
\newtheorem*{Proposition*}{Proposition}

 \newtheorem{Note}[Proposition]{Note}

\newcommand {\z}{{\noindent}}

\def\No{{\bf No}}
\def\N0{\mathbf{No_0}}

\def\XXint#1#2#3{{\setbox0=\hbox{$#1{#2#3}{\int}$}
     \vcenter{\hbox{$#2#3$}}\kern-.5\wd0}}

\def\({\left(}
\def\){\right)}
\def\ge{\geqslant}
\def\le{\leqslant}
\def\epsilon{\varepsilon}
\def\DD{\mathbb{D}}
 
\def\CC{\mathbb{C}}
 \def\RR{\mathbb{R}}
 \def\NN{\mathbb{N}}
\def\ZZ{\mathbb{Z}}
\def\QQ{\mathbb{Q}}
\def\Re{\mathrm{Re\,}}
\def\re{\mathrm{re\,}}
\def\re{\mathrm{re\,}}

\def\bfy{\mathbf{y}}
\makeindex

\def\No{\mathbf{No}}
\title[Integration on the surreals]{Integration on the surreals: A Conjecture of Conway, Kruskal and Norton.}
\author {Ovidiu Costin}
\address{Mathematics Department\\The Ohio State University\\231 w 18th Ave\\Columbus 43210\\costin@math.ohio-state.edu, corresponding author.}
\author {Philip Ehrlich}
\address{Department of Philosophy, Ohio University\\ Athens 45701\\ehrlich@ohio.edu}
\author {Harvey M. Friedman}
\address{Distinguished University Professor of Mathematics, Philosophy, and  
Computer Science, Emeritus, The Ohio State University. Mathematics Department\\The Ohio State University\\231 w 18th Ave\\Columbus 43210\\friedman@math.ohio-state.edu}
\begin{document}
\maketitle

\begin{abstract} 
  In 1976 Conway introduced the \emph{surreal} number system $\No$, containing the reals, the ordinals, and numbers such as  $ - \omega $,  $1/\omega$,  $\sqrt \omega$ and $\ln\omega$. $\No$  is a real closed ordered field, with much additional structure. Surreal theory is conveniently developed within the class theory NBG, a conservative extension of ZFC.

 A longstanding aim has been to develop analysis on $\No$ as a  powerful extension of ordinary analysis on $\RR$. This entails finding a natural way of extending important functions $f:\RR
 \to\RR$ to functions $f^*:\No \to\No$, and naturally defining
 integration on the $f^*$. The usual square root,
  $\log:\RR\to \RR$, and $\exp:\RR\to \RR$ were naturally  extended to $\No$ by
 Bach, Conway, Kruskal, and Norton, retaining their usual
 properties. Later Norton also proposed a treatment of
 integration, but  Kruskal discovered flaws.
 The search for natural extensions from $\RR$ to $\No$, and natural
 integration on $\No$ continues. This paper addresses this and
 related unresolved issues with positive and negative
 results.
 
 In the positive direction, we show that \'Ecalle-Borel
 transseriable functions extend naturally to $\No$, and  an  integral with good properties exists on them. Transseriable functions include semi-algebraic, semi-analytic, analytic, and meromorphic ones as well as solutions of systems of
 linear and nonlinear systems of ODEs with possible irregular
 singularities as in \cite{IMRN}. In particular, most
 classical special functions (such as Airy, Bessel, Ei, erf, Gamma, Painlev\'e and so on) extend naturally (and are integrable) from finite to infinite values of the variable.

 In the negative direction, we show there   is a fundamental obstruction to naturally extending many larger families 
of functions to $\No$ and 
 to defining integration on surreal functions. We show that there are no descriptions of operators  which can be
 proved  within NBG  to have  the basic properties of integration,  even on  highly
 restricted families of real-valued entire functions.

 \noindent
 
 
 \smallskip
 \noindent \textbf{MSC classification numbers: Primary  03E15, 03H05,12J15, 34E05; Secondary 03E25, 03E35, 03E75}
\end{abstract}
\tableofcontents
\section{Introduction}

In his seminal work \emph{On Numbers and Games} \cite{CO1}, J. H. Conway introduced a real-closed field \textbf{No} of \emph{surreal numbers}, containing the reals and the ordinals $0,...., \omega,...$, as well as a great many new numbers, including  $ - \omega $,  $\omega /2$,  $1/\omega $,  $\sqrt \omega $, ${e^\omega}$, $\log \omega $ and $\sin \left (1/\omega  \right)$ to name only a few.  Motivated in part by the hope of providing a new foundation for asymptotic analysis, there has been a longstanding
program, initiated by Conway, Kruskal and Norton, to develop analysis on $\No$  as a powerful extension of
ordinary analysis on the ordered field $\RR$ of reals.  This entails finding a reasonable way
of extending important functions from $\RR$ to $\No$, and to 
define integration on the extensions.

In classical analysis over $\RR$, many transcendental integrals such as $\int_a^x t^{-1}e^tdt$  and special functions that arise as solutions of ODEs have divergent asymptotic series as $x\to\infty$, roughly of the form $\sum c_k x^{-k}$, where the $c_k$ grow factorially. In $\No$, on the other hand, the same sums {\em converge in a natural sense} that we call {\em absolute convergence in the sense of Conway} (see \S3) for all $x\gg 1$. Part of the original motivation for developing surreal integration was the expectation of finding new and more powerful methods for solving such ODEs and summing such divergent series.

  There was initial success with the first part of the program when the square root, 
$\log$, and $\exp$ were extended to \textbf{No} in a property--preserving fashion \cite{vdDE} by Bach, Norton, Conway, Kruskal and Gonshor (\cite[pages 22, 38]{CO2}, \cite[Ch. 10]{GO}). 
Norton's proposed ``\emph{genetic}" definition of integration \cite[page 227]{CO2} addressed the  second part of the program.\footnote{In surreal theory, the field operations on $\No$ as well as the square root, log and exp functions are defined by induction on the complexity of the surreals. Conway has dubbed such definitions ``genetic definitions" \cite[pages 27, 225, 227]{CO2}. In the integration program originally envisioned by Conway, Kruskal and Norton, the functions as well as their integrals were intended to be genetically defined. At present, however, there is no adequate theory of genetic definitions in the literature, though \cite{FO} and \cite{Fornasiero08} contain some useful preliminary remarks. Nevertheless, we will freely refer to various definitions that we define in terms of the complexity of the surreals as ``genetic" (even when they are not inductively defined) since for the most part they conform to the way this notion is informally understood in the literature.\label{f1}} However,  Kruskal subsequently showed the definition is flawed \cite[page 228]{CO2}. Despite this disappointment, the search for a theory of surreal integration has continued \cite{FO}, \cite{RSAS}. Indeed, in his recent survey  \cite[page 438]{SI}, Siegel characterizes the question of the existence of a reasonable definition of surreal integration as ``perhaps the most important open problem in the theory of surreal numbers''.

 Despite being flawed, Norton's integral is highly suggestive: indeed, for a fairly wide range of functions arising from applications, such as solving ODEs, we show that the use of  inequalities coming from transseries and  exponential asymptotics leads to a correct genetic integral. To prove the integral is well defined (in contradistinction to Norton's \cite[page 228]{CO2}) we use \'Ecalle analyzability results, (results that are not required, however, to calculate the integral).  Our constructions also provide a method for solving ODEs in $\No$. However, for more general functions, a substantial obstruction arises which involves considerations from the foundations of mathematics. In particular, we will show that, in a sense made precise, there is no
description which, provably in $\text{NBG}$, defines an integral (of Norton's
type or otherwise) from finite to the infinite domain, even on spaces
of entire functions that rapidly decay towards $+\infty$.
\subsection{Definitions and notation}
We use the space
$T[\mathbb{R}]$  of all functions $f$ whose domain dom$(f)$ is an interval
 in $\mathbb{R}\cup\{-\infty,\infty\}$ and whose range ran$(f)$ is a subset
 of  $\mathbb{R}$.  Empty and singleton intervals are likewise allowed. We define
$\lambda f$ and $f+g$ as usual, where dom$(f+g)$ = dom$(f)\cap$dom$(g)$. For
such intervals $I$ in $\mathbb{R}\cup\{-\infty,\infty\}$, $I^*$ is the corresponding interval in
$\No$ with the same real inf and sup where appropriate and 
$\pm{On}$ in place of $\pm{\infty}$, respectively. $T[\No]$ is the analogous
space of all functions whose domain is an interval of $\No$ whose values lie in $\No$. We
say that $f \in T[\RR]$ is extended by $f^*\in T[\No]$ if and only if
for all $x \in$ dom$(f)$, $f(x) = f^*(x)$, and dom$(f^*) $= dom$(f)^*$.
If $F\in T[\No]$ we denote by $F|_{\RR}$ its restriction to dom$(F)\cap\RR$.

In the positive results in our paper, the integral will be defined from a linear antiderivative. The integral is in a sense a generalization of the Hadamard partie finie integral from infinity. The negative results only need to hold on restricted spaces, and we choose subspaces of 
  \begin{equation}
    \label{eq:defE}
 \mathcal{E}:=\{f:\RR\to\RR:\exists g:\CC\to\CC\; \text{such that} \;g \;\text{is entire and}\; g|{\RR}=f\}. 
  \end{equation}
 These functions have a common domain, and an integral from zero gives rise to a linear antidifferentiation operator. 

Building on results of the first and third authors in \cite{JFA}, the purpose of this paper is to address these two issues:
\vspace{5pt}

1. Which natural families of functions in $T[\mathbb{R}]$ can be naturally lifted to families of $f^*\in T[\No]$, where each $f^*$ extends 
$f$?

2. Is there a natural antidifferentiation operator on the families of  $f^*\in T[\No]$, arising from 1? 

As we show, the existence of  a natural antidifferentiation leads to the existence of a natural integral. This integral follows Norton's scheme.

\vspace{5pt} Both of these questions are treated here in terms of
operators on spaces of real functions.

\begin{Definition}
  \rm{An \emph{extension operator} on $ K \subseteq T[\RR]$ is a function $\mathsf{E}:K \to
T[\No]$ such that for all $f,g \in K$:
\begin{enumerate}[i.]
  \item $\mathsf{E}(f) $ extends $f$;
  \item $\mathsf{E}(\lambda f)$ = $\lambda \mathsf{E}(f)$, $\mathsf{E}(f+g) = \mathsf{E}(f)+\mathsf{E}(g)$;
  \item If $f$ is the real monomial $x^n$ on dom$(f)$, then $\mathsf{E}(f)$ is the
surreal monomial $x^n$ on dom$(f)^*$, for all $n\in\NN$;
  \item If $f\in K$ is the real exponential on dom$(f)$, then $\mathsf{E}(f)$ is the surreal  
exponential on dom$(f)^*$.
\end{enumerate}
}
\end{Definition}

To formulate the appropriate notion of an antidifferentiation operator, we require a generalization of the idea of a \emph{derivative of a function at a point}.

 \begin{Definition}\label{DefDeriv}{\rm 
 Mimicking the usual definition, we  say that $f$ is differentiable at $a$ in an ordered field $V$ if there is an $f'(a)\in V$ such that $(\forall \epsilon >0 \in V)(\exists \delta >0 \in V)$ such that $(\forall x\in V)(|x-a|<\delta\Rightarrow |(f(x)-f(a))/(x-a)-f'(a)|<\epsilon)$. As usual, $f'(a)$ is said to be \emph{the derivative of $f$ at $a$}.
}\end{Definition}

\begin{Definition}\label{Dd2}
  \rm{An \emph{antidifferentiation operator} on $K \subseteq T[\RR]$ is a function
$A:K \to T[\No]$ such that for all $f,g \in K$:
\begin{enumerate}[i.]
  \item $A(f)'$ extends $f$;
  \item $A(\lambda f)$ = $\lambda A(f)$, $A(f+g) = A(f)+A(g)$;
  \item If $f$ is the real monomial $x^n$, then $A(f)$ is the
surreal monomial $x^{n+1}/(n+1)$ on dom$(f)^*$;
  \item If $f\in K$ is the real exponential on dom$(f)$, then $A(f)$ is the
surreal exponential on dom$(f)^*$.
\end{enumerate}
}
\end{Definition}{\rm 
For suitable integrals to exist, we need the ``second half'' of the fundamental theorem of  calculus to hold. This is the motivation for the following strengthening of Definition \ref{Dd2}.
\begin{Definition}\label{DefAf}{\rm  
\emph{A strong  antidifferentiation operator} on $K \subseteq T[\RR]$ is an antidifferentiation operator $A$ such that if $F\in T[\No]$ and $(F|_{\RR})'=f\in K$ exists, then there is a $C\in\No$ such that $A(f)=F+C$.
  }\end{Definition}

Our first group of results show that there are (unnatural) extension and antidifferentiation 
operators on $T[\mathbb{R},\RR]:=\{f\in T[\RR]:\mathrm{dom}(f)=\RR\}$ correctly acting on  finitely many, or even all monomials. For finitely many monomials, the proof is constructive. For infinitely many, the proof uses the classical result that every vector space has a basis extending any given linearly independent set. 
Of course, such a proof does not lead to any reasonably defined examples of extension 
or antidifferentiation operators. 
In fact, our negative results, discussed below, establish that this drawback is 
unavoidable--even for much more restrictive $K$'s. 

\vspace{5pt}
Our second group of results are negative, and establish, in several related senses, 
that there are no reasonable extension or antidifferentiation operators on the 
space of all real-valued entire functions, even if we restrict ourselves to those with slow 
growth. Specifically, we show that there are no descriptions, which, provably, uniquely define such extensions or antidifferentiations, even in the presence of the axiom of choice. Our negative results show that, in order to naturally 
extend families of real-valued entire functions past $\infty$ into the infinite surreals, 
critical information about the behavior of the functions at $\infty$ is required. For families of 
real-valued entire functions characterized by growth rates only, we show that polynomial 
growth forms a threshold for extendability (in a sense made precise). As a consequence of Liouville's 
theorem, this entails that such classes of functions consist entirely of polynomials. 
Analogous results are also shown to apply to antidifferentiation.

Our third group of results construct natural  extension and strong antidifferentiation operators 
with very nice features. Proposition \ref{existint} below shows that the latter suffices for the existence of an actual integral. In fact, the antiderivative defines an {\em  integral} with essentially all the properties 
of the usual integral. Moreover, when restricted to $\RR$, the antiderivative is a proper extension of the Hadamard finite part at $\infty$; see \cite{JFA}. This group of results is applicable to sets of functions that, at $\infty$, are
meromorphic,  semi-analytic, Borel summable, or lie in a class of \'Ecalle-Borel
transseriable functions; see \cite{Book},\cite{Ecalle2}. As such, the sets of functions to which our positive results apply include most sets of standard classical functions. We argue that for all ``practical purposes" in applied analysis, integrals and 
natural extensions from the reals to the surreals with good properties exist. In particular, most
 classical special functions (such as Airy, Bessel, Ei, erf, Gamma, Painlev\'e,...) extend naturally (and are integrable) from finite to infinite values of the variable.\footnote{Integration, for functions with convergent expansions, has been studied in the context of the non-Archimedean ordered field of left-finite power series with real coefficients and rational exponents in \cite{SB} and \cite{S}.}

In \S \ref{XX}, Logical Issues, we discuss appropriate
formal systems in which to cast the results. We identify
two main approaches to formalization: literal and classes.
NBG is the principal system for the classes approach, which
is the approach used in the body of the paper. For the
literal approach, we use ZCI = ZC + ``there exists a
strongly inaccessible cardinal". Both approaches have their
advantages and disadvantages. The underlying mathematical
developments are sufficiently robust as to be unaffected by
these logical issues. As a consequence, the reader
unconcerned with logical issues need not read \S \ref{XX}.

\vspace{5pt}

This paper is the result of an interdisciplinary collaboration making use of the 
first author's expertise in  standard and non-standard analysis,  exponential asymptotics and  Borel summability, the 
first and second authors' expertise in surreal numbers, and the third author's expertise 
in mathematical logic and the foundations of mathematics.

\subsection*{Organization of the paper}
The paper relies on general notions, constructions and results from the theory of surreal numbers. An overview of those that are most relevant to this paper are given in \S\ref{S2}. For the most general positive results, the paper also uses  transseries and \'Ecalle-Borel summability. These are reviewed in \S\ref{Sec5.3}.

\section{Main results}

Our first set of results prove the existence of surreal
antiderivatives within NBG (which includes the axiom of choice).
Another set of results show that  $\text{NBG}^-$ + DC cannot prove their
existence, even in sharply weakened forms.  Here $\text{NBG}^-$ is NBG without any form of the axiom of choice, and DC is dependent choice, which is an important weak
form of the axiom of choice; see \S \ref{XX}. Thus, in fairly general
settings, the existence of antidifferentiation operators in the
surreals is neither provable nor refutable in
NBG$^-$ + DC.

We also prove that there are no explicit descriptions that
can be proved to uniquely define such operators, even in NBG. On the other hand, we show, in NBG$^-$, that 
genetically  defined integrals
 (see footnote \ref{f1}) exist for functions
arising  in all ``practical applications.''

\subsection{Unnatural positive results}\label{Un+}
The results in this section  do not lead to natural operators or even natural functionals.  Specifically, while these operators or functionals  assume, by construction,  the expected values on simple functions, for general functions their existence relies on the axiom of choice and as such their values for specific functions are intrinsically ad hoc.  
\begin{Theorem}\label{Ad-hoco}{\rm 
  Let $\{1,x,...,x_n\}$ be the set of monomials of degree $<n+1$, and $W$ be the linear space generated by them (i.e., the polynomials of degree $<n+1$). Then there is an antidifferentiation operator $A$ on $T[\RR,\RR]$ such that $A(x^{j-1})=x^{j}/j, j=1,...,n+1$.
}\end{Theorem}
\begin{Note}
  {\rm  The proof of Theorem 4 is carried out as an explicit
construction in NBG$^-$.
Finitely many other
natural functions with natural derivatives can be incorporated.}
\end{Note}

The next result uses the axiom of choice.

\begin{Theorem}\label{Ad-hoc1}{\rm 
There is an antidifferentiation operator $A$ on $T[\RR,\RR]$ such that $A(x^{j-1})=x^j/j \;\text{for each} \;j\in\ZZ^+$. More generally, if $W_1\subset T[\RR,\RR]$ is  a space of  functions  that naturally extend to ${\bf No}$ in such a manner that $ F^*(x)-F^*(0)$ is a natural linear  antiderivative of $f(x)$, then $A(f)(x)=F^*(x)-F^*(0)$. Moreover, $W_1$ would include the real exp, $1/(x+1)$ and other elementary functions.}

\end{Theorem}

\begin{Theorem}
 {\rm
 There exist $2^\mathfrak{c}$ extension operators $\mathsf{E}$ on the set of all functions $f:\RR \to\RR$ which preserve polynomials and other known functions. 
}\end{Theorem}
\begin{proof}
{\rm This follows from Theorem \ref{Ad-hoc1} by letting  $\mathsf{E}(f):=(A(f))'$. 
}\end{proof}

\begin{Note}
{\rm   The unnatural positive results above which do not involve exponentials extend to general non-Archimedean fields.
}\end{Note}

\vspace{0.3cm}
\subsection{Obstructions}\label{S2.2O} $ $\\

The next set of results are negative. We present different types of obstructions to integration with the aim of clarifying the nature of these obstructions. The first group of results (Theorem \ref{TT18} - Theorem \ref{nonintext}) are targeted primarily (though not exclusively) to the proposed Conway-Kruskal-Norton program of surreal integration based on ``genetic" definitions; see footnote \ref{f1}. These results suggest that for their integration program on $\No$ to succeed,  even for rapidly decaying entire functions, stringent conditions governing the behavior of the functions at $\infty$ are needed. The remaining negative results hold for any non-Archimedean ordered field $V$ with specified properties.

Throughout this section, the space $\mathcal{E}$ is defined as in \eqref{eq:defE}, $V$ (which may be a set or a proper class) is a non-Archimedean ordered field extending $\RR$, and $\omega$ is a positive infinite element of $V$. If $V=\No$, then $\omega$ may be taken to be $\{\NN|\}$. 

We start with a standard definition.
\begin{Definition}
  \rm A set of reals is Baire measurable if and only if its symmetric difference with some open set is meager. This is often called the ``property of Baire.'' \end{Definition}

Let $\mathcal{E}_b=\left \{ f\in\mathcal{E}:\|f\|_{\infty}<\infty\ \right \}$, where, as usual, $\|f\|_{\infty}=\sup_{x\in\RR}|f|$. For Theorems \ref{TT18} and \ref{ZFHB} below we could further restrict the space to that of entire functions of exponential order at most two; moreover, compatibility with all monomials is not required.

We remind the reader that a {\bf Banach limit} is a  linear functional $\varphi$ on the space $\ell^{\infty}=\{s:\NN\to\CC|\sup_n |s_n|<\infty\}$ such that $\varphi(Ts)=\varphi(s)$ where $T(s_0,s_1,...)=(s_1,s_2,...)$ and $\varphi(s)\in [\inf_{n\in\NN}s_n,\sup_{n\in\NN}s_n]$, if ran$(s)\subset\RR$. We also remind the reader that NBG$^-$, which is conservative over ZF, includes no form of choice.

The following definition makes use of the conceptions introduced in Definitions \ref{DefExtf} and \ref{DefAf} from \S \ref{Un+}.
\begin{Definition}
{\rm  $\mathfrak{E}^+_{\omega}:=\{f\in\mathfrak{E}_{\omega}:(h\in \mathcal{E}_b,h\ge 0)\Rightarrow f(h)\ge 0\} $. $\textgoth{A}^+_{\omega}$ is similarly defined.}
\end{Definition}
\begin{Theorem}\label{TT18}
  {\rm NBG$^-$  proves that if $\mathfrak{E}^+_{\omega}\cup \textgoth{A}^+_{\omega} \ne \emptyset$, then Banach limits exist.}
\end{Theorem}
\begin{proof}
The proof for $\textgoth{A}^+_{\omega}$ is a corollary of the following one about $\mathfrak{E}^+_{\omega}$. It is enough to define a Banach limit $\varphi$ on real sequences. Let  $l_s$ be $0$ for $x\le0$ and the linear interpolation between $0$ and $s_0$ on $[0,1]$ and between $s_n$ and $s_{n+1}$ on $[n+1,n+2]$ $\forall n\in\NN$. Now, $\forall s\in\ell^{\infty} $ and  $\forall \epsilon>0$,  $\exists  f_{s;\epsilon}\in \mathcal{E}_b$ (obtained by mollifying $l_s$) such that $\|f_{s;\epsilon}-l_s\|_{\infty}<\epsilon$.  Let $F_{s;\epsilon}(x):=\int_0^x f_{s;\epsilon}$ and $L_s(x):=\int_0^x l_s$. Note that $\forall x\in\RR^+$ we have (*): $|F_{s;\epsilon}(x)-L_s(x)|\le \epsilon x$. 
Let $\mathsf{E}\in\mathfrak{E}_{\omega}^+$. Define $\varphi$ on $\ell^\infty$ by  $\varphi(s):=\lim_{\epsilon\to 0}\, \re\left(\,\omega^{-1}\mathsf{E}(F_{s;\epsilon})(\omega)\right)$, where re$(x)$ denotes the real part of a surreal number $x$ written in normal form (see  \S\ref{S2}); this exists by (*). By mimicking Robinson's non-standard analysis construction (cf. \cite[page 54]{Loeb}), it is easy to check that  $\varphi$ is a Banach limit. (The calculations are straightforward; see  \S\ref{Sn4.1} for the details.)\end{proof}

\begin{Theorem}\label{ZFHB}
  {\rm 
    \begin{enumerate}
    \item   NBG$^-$ proves that the existence of Banach limits implies the existence  of a subset of $\RR$ not having the property of Baire.
\item  NBG$^-$+ DC does not prove the existence  of a subset of $\RR$ not having the property of Baire.
\item  $\mathfrak{E}^+_{\omega}\cup \textgoth{A}^+_{\omega} \ne \emptyset$ is independent of NBG$^-$+ DC.
    \end{enumerate}
}
\end{Theorem}
\begin{proof}
(1) ZF and, thus, NBG$^-$ proves that the existence of Banach limits implies the existence  of a subset of $\RR$ not having the property of Baire; see \cite[pages 611-612]{Hand} and \cite{Kharazishvili}. 

(2) ZFDC does not prove the existence  of a subset of $\RR$ not having the property of Baire; see \cite{Solovay} and  \cite{Shelah}. Since NBG$^-$+ DC is a
conservative extension of ZFDC, (2) follows.

(3) This follows from (2) and Theorem \ref{TT18}.
\end{proof}
The above results indicate the nature of the obstructions to the existence of a general ``genetic'' integral (see footnote \ref{f1}): 
it is implicit in the (informal) description of a such an integral that it should be explicitly constructed in NBG$^-$ and, hence, \emph{without any form of choice.}  However, such an integral exists at most on those functions whose behavior {\em at} $\infty$ ensures that an {\em explicitly constructed} Banach limit  exists, as in Theorem \ref{TT18}. Even this condition is insufficient, as we shall now see. The class of functions for which genetic definitions exist is in fact highly restricted. See \S\ref{Sconstr}.

The next sequence of results shows that the obstructions persist even if we gradually impose more constraints on the functions, until the constraints are so restrictive that the remaining functions are simply polynomials.

In the following definition, $\mathcal{E}_1$ is a space of real functions,  {\em exponentially decaying on $\RR^+$},  that extend to  entire functions of exponential order one.  In classical analysis, growth conditions  (certainly exponential decay) at a singular point of an otherwise smooth function would ensure integrability  at that point. This however  is not the case in the surreal world.

\begin{Definition}
{\rm   Let
$\displaystyle  \mathcal{E}_1:=\left\{f \in \mathcal{E}: \|f\|:=\sup_{z\in\CC}|e^{-2|z|}f(z)|+\sup_{x\in\RR^+}|e^{x/2}f(x)|<\infty\right\}$. $\forall n\in\ZZ^+$, let $f_n(z)=  n e^{-z}(1-e^{-z/n})$  and  $\mathcal{E}_e\subset \mathcal{E}$  be the  closed subspace   generated by $\{f_n\}_{n\in\ZZ^+}$.
}\end{Definition}
\begin{Lemma}\label{LBanach}{\rm 
  $\mathcal{E}_1$  is a Banach space, and $\forall n\in\ZZ^+$, $\|f_n\|<2$. In addition, $\mathcal{E}_e$ is a  separable Banach space\footnote{Note that $\mathcal{E}_1$ is {\em not} separable. Indeed, for any $u_1\ne u_2\in S^1$, $\|f_{u_1}-f_{u_2}\|=1$, where $f_{u_i}=e^{\textstyle u_i z}$; the same holds for $f_{u_i}+\overline{f_{u_i}}\in \mathcal{E}_1$ if $\arg u_i\in I=(\frac34\pi,\frac54\pi)$, and $I$ is of course uncountable.} and thus a Polish space.
}\end{Lemma}
\begin{proof}
 For $K>0$, let $\DD_K=\{z:|z|<K\}$.    On any $\overline{\DD_K}$,  $\|\cdot\|$ is manifestly equivalent  to the sup norm. A uniform limit on $\overline{\DD_K}$   of functions in $\mathcal{E}$  is analytic in $\DD_K$ and real valued on $[0,K)$, thus $\mathcal{E}$ is a Banach space. The rest  is calculus.
\end{proof}
The functionals  in $\mathfrak{E}$ ($\textgoth{J}$, respectively) below have properties expected of any extension to $\omega$ (integration on $[0,\omega]$, respectively).
\begin{Definition} {\rm 
  $\mathsf{E}\in\mathfrak{E}$ if $\mathsf{E}$ is linear on $\mathcal{E}_e$, and  $\forall a\in \RR$, $\mathsf{E}(e^{ax}) =e^{a\omega}$ when $e^{ax}\in \mathcal{E}_e$.  $\mathsf{J}\in\textgoth{J}$ if $\mathsf{J}$ is linear on $\mathcal{E}_e$, and  $\forall a\in \RR$, $\mathsf{J}(e^{ax}) =a^{-1}(e^{a\omega}-1)$  when  $e^{ax}\in \mathcal{E}_e$.
}\end{Definition}
\begin{Lemma}\label{Ldiscont}
{\rm  If $\mathfrak{E}\cup \textgoth{J} \ne \emptyset$, then there is a \emph{discontinuous} homomorphism  $\varphi:\mathcal{E}_e\to\RR$}. 
\end{Lemma}
\begin{proof}
Assume $\mathsf{E}\in \mathfrak{E}$. Define  $\varphi:\mathcal{E}_e\to\RR$  by $\varphi(f)=\re(e^{\omega} \mathsf{E}(f))$, where re$(e^{\omega} \mathsf{E}(f))$ is the real part of the normal form of the surreal number $e^{\omega} \mathsf{E}(f)$; see \S \ref{S2}. Clearly, $\varphi$ is linear. Moreover, since $\lim_{n\to\infty}\|n^{-1}f_n\|= 0 $ and  $\varphi(n^{-1}f_n)=1$,  $\varphi$ is discontinuous. Similarly, assuming that  $\mathsf{J}\in \textgoth{J}$, $\varphi(f)=\re (e^{\omega} \mathsf{J}(f))$ is discontinuous.
\end{proof}
 In virtue of Pettis's Theorem \cite{Pettis}, Solovay's results \cite[\S 4, p. 55]{Solovay}, Thomas-Zapletal's Theorem 1.4 \cite{Thomas} and the equiconsistency of ZFDC and  NBG$^-$+ DC, the following holds.
\begin{Theorem}[\cite{Thomas}]\label{Thomas}{\rm The following is consistent with NBG$^-$ + DC. For
all Polish groups $G,H$, every homomorphism $\varphi:G\to H$ is continuous.
}
\end{Theorem}
{\begin{Corollary}\label{ZFDC1}
{\rm  $\mathfrak{E}\cup \textgoth{J}=   \emptyset$ is consistent with NBG$^-$+ DC. 
}\end{Corollary}
\begin{proof}
  This is an immediate consequence of Lemmas  \ref{LBanach} and \ref{Ldiscont}, and Theorem \ref{Thomas}.
\end{proof}
\begin{Theorem}
{\rm  $\mathfrak{E}\cup \textgoth{J}\ne   \emptyset$ is {\em independent of} NBG$^-$+ DC. 
}\end{Theorem}
\begin{proof}
  This is obtained by combining Corollary \ref{ZFDC1} with Theorem \ref{Ad-hoc1}.
\end{proof}
In virtue of the above, it is clear that if there is a genetic definition of integration on $\No$, its validity cannot be established in NBG$^-$+ DC. 
This being the case, it is intuitively obvious that by supplementing NBG$^-$ with either the axiom of choice, the axiom of global choice or, say, the Continuum Hypothesis cannot help prove that a {\em concrete} integration definition has the intended properties. 

\begin{Note}
 {\rm  It is clear from the proofs that the results above would be unchanged if $e^{-|z|}$ is replaced by some $e^{-|z/n|}, n\in\ZZ^+$, and in fact similar  results carry over to other types of bounds or decay.  For functions of half exponential order, the family $\cosh(\sqrt{z}/n)$ can be used instead of $e^{-z/n}$; for arbitrary types of growth, the argument requires several steps, since special function are to be avoided, and as such, different foundational tools need to be used. In any case, any growth bounds, short of polynomial bounds,  yield the same result.   Polynomial bounds entail that the functions are polynomials, in which case integration exists trivially. Furthermore, these obstructions continue to hold for more general non-Archimedean fields, and we provide the results and detailed proofs for the more general case.}  
\end{Note}

\subsubsection{Further obstructions} 
We now obtain sharper negative results based on
descriptive set-theoretic tools adapted to the questions at hand. By
sharper negative results, we mean negative results that require fewer
properties of the integral, allow for ``nicer'' functions and which, in
some settings, provide both necessary and sufficient conditions for
the existence of extension or antiderivative functionals. Theorem \ref{nonintext}
is a consequence of the core descriptive set-theroetic result, Theorem \ref{T1.1.}, that does not rely on any substantial analytic arguments. Theorem
\ref{2.1T71} is a deeper consequence of Theorem \ref{T1.1.} showing that, in the class of
entire functions bounded in some weighted  $L^\infty$ norm in $\CC$, integrals
exist if and only if the weight allows only for polynomials.

We begin with the definitions of $\theta$-\emph{extensions}  and $\theta$-\emph{antiderivatives} in the setting of  ordered  fields.

\begin{Definition}\label{DefExtf}{\rm 
A   $\theta$-extension from $K\subseteq T[\RR,\RR]$ into ${\bf No}$ is a function $e_{\theta}:K\to {\bf No}$ such that
   \begin{enumerate}[i.]
   \item $e_{\theta}$ is linear;
   \item If $f\in K$ is the monomial $x^n$,  then $e_{\theta}\, (f)=\theta^n$.
   \end{enumerate}
We denote by $\mathfrak{E}_{\theta}$ the collection of $\theta$-extensions.
}\end{Definition}
\begin{Definition}\label{DefAf}{\rm  
A   $\theta$-antidifferentiation from $K\subseteq T[\RR,\RR]$ into ${\bf No}$ is a function $a_{\theta}:K\to {\bf No}$ such that
   \begin{enumerate}[i.]
   \item $a_{\theta}$ is linear;
   \item If $f\in K$ is the monomial $x^{j-1},j>0$,  then $a_{\theta}\, (f)=\theta^{j}/j$.
   \end{enumerate}
We denote by $\textgoth{A}_{\theta}$ the collection of $\theta$-antidifferentiations.
}\end{Definition}
There is no immediate relation between $\theta$-extensions and $\theta$-antiderivatives.

\begin{Theorem}\label{nonintext} {\rm  
  \begin{enumerate}[i.]
  \item NBG$^-$+ DC proves the following. If there exists an ordered field $V$ extending  $\RR$, a positive infinite $\theta\in V$,  and a  $\theta$-extension  or a  $\theta$-antiderivative functional
from $\mathcal{E}$  into $V$,  then there is a set of
reals that is not Baire measurable.
\item  NBG$^-$+ DC does not prove that there exists an ordered field $V$ extending  $\RR$, a positive infinite $\theta\in V$,  and a  $\theta$-extension or  a  $\theta$-antiderivative functional
 from $\mathcal{E}$  into $V$.
\item \label{itmt3} There are no three set-theoretic descriptions with parameters such that NBG proves the
following. There is an assignment of real
numbers to the parameters used in all three,
such that the first description uniquely
defines an ordered field extending $\RR$, the
second description uniquely defines a
positive infinite $\theta\in V$, and the third
description uniquely defines a $\theta$-extension 
functional  or a  $\theta$-antiderivative functional 
from $\mathcal{E}$  into $V$.

Part \ref{itmt3}  also holds for extensions of NBG by
standard large cardinal hypotheses.
  \end{enumerate}

}\end{Theorem}
\begin{Note} {\rm  In Theorem \ref{nonintext}, ii follows immediately from i, and the well
known fact that NBG$^-$ + DC + ``every set of reals is has the
Baire property'' is consistent; see \cite{Solovay} and  \cite{Shelah}).
}
\end{Note}

\subsection{More general negative results}\label{S2.3O} We explore further spaces without regularity, and obtain stronger versions of the results above, in more generality than $\No$.

We  view {\em any}  $W:[0,\infty]\to \RR^+$ as a \emph{weight}.
\begin{Definition}[Growth class $W$]\label{GC}
	{\rm  Let $W:[0,\infty)_\RR\to \RR^+$. Define
			\begin{equation}
			\label{eq:grW}
		\mathcal{E}_W=\{f\text{ entire }, f(\RR)\subseteq\RR:	\sup_{z\in\CC}|f|/W(|z|)\le C<\infty\}.
			\end{equation}
}	
\end{Definition}
We now state an important if and only if result that provides a dichotomy. Our formulation requires the notion of a continuous weight given by an \emph{arithmetically  presented code}.
\begin{Definition}
   {\rm Let $W :\RR\to\RR^+ $ be a continuous weight. We say that $E\subset  \ZZ^4$  codes  $W$ if and only if $\forall (a,b,c,d) \in \ZZ^4$,  $a/b < W(c/d)$ just in case $(a,b,c,d) \in  E$. An arithmetic presentation of $E \subset  \ZZ^4$ takes the form $\{(a,b,c,d) \in \ZZ^4: \phi\}$, where $\phi$ is a formula involving $\forall, \exists, \neg, \vee , \wedge, +,-,\cdot,<,0,1,$  variables ranging over $\ZZ$,  with at most the free variables $a,b,c,d$. 
}\end{Definition}

Standard elementary and special functions with rational parameters are continuous functions that can be given by
arithmetically presented codes. Also, compositions of continuous functions that are given by arithmetically presented codes are continuous functions that are given by  arithmetically presented codes.\footnote{Note that here and in \cite{JFA}, we consider only arithmetically presented continuous functions. There is a broader notion of ``arithmetically
presented function" that includes many discontinuous Borel measurable
functions, which we do not need here or in \cite{JFA}. However, in \cite[page 4738, paragraph 3]{JFA}, the first and third authors erroneously confused the two notions.
That entire paragraph should be replaced with the
paragraph to which this note is affixed.
}

 \begin{Theorem}\label{2.1T71}
		{\rm  
                 
                  Let $W$ be a continuous weight given by an arithmetically presented code.\footnote{More precisely, we start with an arithmetically presented $E$ which,
provably in ZFC (or equivalently in ZF, NBG, NBG$^-$) codes a continuous
weight W$.$} The following are equivalent.

                  \begin{enumerate}[(a)]
                  \item NBG$^-$ proves that $\mathcal{E}_W$ consists entirely of polynomials from $\RR$ into $\RR$. 
                  \item There are three set-theoretic descriptions with parameters such
that NBG proves the following. There is an assignment of real
numbers to the parameters used in all three,
such that the first description uniquely
defines an ordered field extending $\RR$, the
second description uniquely defines a
positive infinite $\theta\in V$, and the third
description uniquely defines a $\theta$-extension 
functional  or a  $\theta$-antiderivative functional 
from $\mathcal{E}_W$  into $V$. 
                  \end{enumerate}    
The equivalence holds if NBG$^-$  is replaced by NBG or if both  NBG$^-$  and NBG are replaced by any common extension of NBG with standard large cardinal hypotheses.

 }
\end{Theorem}
The next question is: can we integrate (or extend past $\infty$) real-analytic functions for which the real integral $\int_1^\infty |f|<\infty$, or that decrease much faster, to ordered fields such as $\No?$ To help prepare the way for answering this question we require a series of definitions, some of which make use of our generalization of the notion of a derivative of a function at a point; see Definition \ref{DefDeriv}.

 \begin{Definition}
  {\rm An \emph{exponentially adequate ordered field} is an ordered
field $V$  extending $\RR$ together with an order preserving mapping $\exp=(x\mapsto e^x)$ (from the additive group of $V$ onto the multiplicative group of positive elements of $V$) with the properties $e^x\gg x^n$ for all positive infinite $x$ and $(e^x)'=e^x$}.  
\end{Definition}
\begin{Definition}
  {\rm For each $m\in\ZZ^+$, let $\exp_m$ be the $m$-th compositional iterate of $\exp$,  $W_m:\RR\to \RR^+$ be the weight $1/\exp_m$, and $\mathcal{A}_{W_{m}}=\{f\in \mathcal{E}:	\sup_{\RR^+}|f|/W_{m}<\infty\}$. 
 } \end{Definition}
 
 The functions in the $\mathcal{A}_{W_{m}}$``decrease superexponentially''.

\begin{Definition}\label{DefWexp}
{\rm  Exponentially adequate $\theta$-extension functionals and
exponentially adequate $\theta$-antidifferentiation functionals are
$\theta$-extension functionals and $\theta$-anti\-differentiation functionals
obeying the following additional condition for all $m,n \in \ZZ^+$. 

If $f=x^ne^{-W_m}W_m^{'}(1-n/(xW_m^{'}))  $, then $e_{\theta,V}(f)=f(\theta)$ and $a_{\theta,V}(f)=-\theta^ne^{-W_m(\theta)} $. This  corresponds to the antiderivative of $f$ without constant term.
}\end{Definition}

\begin{Theorem}\label{expogr.}{\rm  i. NBG$^-$+ DC proves the following. If there exists an exponentially adequate ordered field $V$, a positive
infinite $\theta$ in $V$, and an exponentially adequate $\theta$-extension or
$\theta$-antiderivative functional from some $\mathcal{A}_{W_{m}}$ into $V$, then there is
a set of reals that is not Baire measurable.

  \item ii. NBG$^-$+ DC  does not prove that there exists an exponentially adequate
ordered field $V$, a positive infinite $\theta$ in $V$, and an
exponentially adequate $\theta$-extension or $\theta$-antiderivative
functional from some $\mathcal{A}_{W_{m}}$ into $V$.

  \item iii. There are no three set-theoretic descriptions with parameters such that NBG 
proves the following. There is an assignment of real numbers to the
parameters used in all three, such that the first description uniquely
defines an exponentially adequate ordered field, the
second description uniquely defines a positive infinite $\theta$ in $V$,
and the third description uniquely defines an exponentially adequate
$\theta$-extension functional or $\theta$-antiderivative function from some
$\mathcal{A}_{W_{m}}$ into $V$.
\noindent

Part iii also holds for extensions of NBG by standard large cardinal hypotheses.

\vspace{0.08cm}
}\end{Theorem}

\subsection{Positive results about extensions and integrals in $\No$ for functions with regularity conditions at infinity}\label{Sconstr} $ $\\

In the previous sections we have obtained ad hoc extension and antidifferentiation operators and showed that there are no natural ones unless there are restrictions on the behavior of the functions at the endpoints.

In this section, we show that when certain natural restrictions are present, genetically defined (see footnote \ref{f1}) integral operators, and extensions with good properties do exist. Integrals are defined after constructing linear antidifferentiation operators, see Definition \ref{Dd2}; also see Proposition \ref{existint} and Note \ref{exp} below. Moreover, as we will see, for all special functions for which a genetically defined integral exists, its value at $\omega$ determines the transseries of the function, which in turn completely determines the function.  The domain of applicability of these genetically defined
integral operators, and extensions includes a wide range of functions arising in applications.
\begin{Note}
  {\em In the following, without loss of generality, we  assume that  the functions of interest are defined on $(x_0,\infty)$, where $x_0$, which is assumed to be in $\RR \cup\{-\infty\}$, may depend on the function and we seek extensions and integrals thereof to $\{x\in\No:x>x_0\}$.
See \S\ref{Srestr} for details.}\end{Note}

\subsubsection{\bf The family $\mathcal{F}$ of functions for which we obtain ``good'' extensions and integrals}}\label{3.4.1.} Our positive results apply when there is complete information about the behavior of functions at $+\infty$ in the sense described above. We denote the family of such functions $\mathcal{F}$. It is convenient to first treat a proper subclass $\mathcal{F}_a$ of $\mathcal{F}$ that is simpler to analyze than $\mathcal{F}$ itself.

\begin{enumerate}[(a)]
\item $\mathcal{F}_a$ consists of  functions which have  convergent Puiseux-Frobenius power series in integer or noninteger powers of $1/x$ at $+\infty$. Namely, for some $k\in\QQ$,  $r\in\RR$, and all    $x>r$, we have
\begin{equation}\label{PSeq}
f(x)=\sum_{j=-M}^{\infty}c_k x^{-j/k}.
\end{equation}
Such is the case of functions which at $+\infty$ are semialgebraic, analytic, meromorphic and semi-analytic, to name some of the most familiar ones. By \cite{VdDries}, semi-analytic functions at $+\infty$ have convergent Puiseux series in powers of $x^{1/k}$ for some  $k\in\ZZ^{+}$. We could allow \eqref{PSeq}  to also contain exponentials and logs, if convergence is preserved.
\item  The class $\mathcal{F}$ is much more general. It consists of  functions which, after changes of variables, have for $x>x_0$ Borel-\'Ecalle transseries  of the form
\begin{equation}\label{trans2}
  \tilde{T}=  \sum_{\bfk\ge \bfk_0,l\in\NN}c_{\bfk,l}x^{\boldsymbol{\beta}\cdot \bfk}e^{-\bfk \cdot \boldsymbol{\lambda} x}x^{-l},
  \end{equation} where $\bfk,\bfk_0\in\ZZ^n, \boldsymbol{\beta}=(\beta_1,...,\beta_n),\boldsymbol{\lambda}=(\lambda_1,...,\lambda_n)$  \text {and} $c_{{\bf {k}},l}, \beta_j,  \lambda_j \in \RR$.   We assume $ k_j\lambda_j$ be present for infinitely many $k_j$ {\bf only}
 if $\lambda_j>0$. For a technical reason---being able to employ Catalan averages---we assume as in \cite{IMRN} that all $\beta_i\le 1$; it can then be arranged that  $\beta_i\in (0,1]$.   We  arrange  that  the critical \'Ecalle time is $x$ (\cite{IMRN,Book, Duke, Ecalle1,Ecalle2}).\footnote{We  could allow  the constants $c_{\bfk,l}, \boldsymbol{\beta}$ and  $\boldsymbol{\lambda}$ to be complex-valued--and then we would end up with a class of {\em sur-complex functions}.    We would then require $\Re\beta_j<1$  (\cite{IMRN,Duke}), it being understood that $\Re$ denotes the real part in the sense of complex analysis (which should not be confused with ``re'' which denotes the real part of a surreal number written in normal form; see \S\ref{S2}).}

$\mathcal{F}$ contains the solutions of linear or nonlinear systems of ODEs or of difference equations for which, after possible changes of variables, $+\infty$ is at worst an irregular singularity of Poincar\'e rank one. We will impose still further technical restrictions, to keep the analysis simple. The setting is essentially that of \cite{IMRN,Duke} for ODEs and \cite{Braaksma} for difference equations. More generally, $\mathcal{F}$  includes elementary functions and all named functions arising in the analysis of ODEs, difference equations and   PDEs. The named classical functions in analysis that satisfy some differential or difference equation such as Airy Ai and Bi, Bessel $H,J,K,Y$,  $\Gamma$, $_2F_1$, and so on  are real analytic in $x$ for $x>0$ and belong  to $\mathcal{F}$.
\end{enumerate}

\begin{Proposition}[\emph{Existence of an integral operator}]\label{existint}
{\rm Let $A$ be a strong antidifferentiation operator on $K$. Then there exists an \emph{integral operator} on $K$, meaning  a function of three variables, $x,y\in\No$ and $f\in K$, denoted as usual $\int_x^y f$, with the following properties:  
\begin{enumerate}[(a)]
 \item $\displaystyle \left(\int_a^x f\right)'=f\ \ \forall a,x \in \mathrm{dom}(f)$;
 \item $\displaystyle \int_a^b(\alpha f+\beta g)=\alpha \int_a^b f+\beta \int_a^b g\ \ \ \forall  a,b\in \mathrm{dom}(f)\cap \mathrm{dom}(g), (\alpha,\beta)\in\RR^2$;
\item $\displaystyle \int_a^b f' =f(b)-f(a)$;
 
  \item  $\displaystyle \int_{a_1}^{a_2} f+\int_{a_2}^{a_3}f=\int_{a_1}^{a_3} f\ \ \  \forall a_1,a_2,a_3 \in \mathrm{dom}(f)$;
 
  \item  $\displaystyle\int_a^b f'g=fg|_a^b-\int_a^b fg'\ \ \text{if $f,g$ are differentiable and} \; a,b\in \mathrm{dom}(f)\cap \mathrm{dom}(g)$;
  \item \label{itemT5}   $\displaystyle\int_a^x f(g(s))g'(s)ds=\int_{g(a)}^{g(x)}f(s)ds,\\ $
whenever $g\in K$ is differentiable, $a,x\in \mathrm{dom}(f)$, \text{and} $g([a,x])\subset \mathrm{dom}(f)\cap \mathrm{dom}(Af)$;

\item $f\ge 0$ and $b>a$ are in dom$(f)$  imply $\int_a^b f\ge 0$.
  \end{enumerate}
  
}\end{Proposition}
\begin{proof}
  Define $\int_{x}^{y} f=A(f)(y)-A(f(x))$. The properties follow straightforwardly. (c) is the definition of a strong antiderivative, and (c), (e) and (f)  are equivalent. If $f\ge 0$, then $(A(f))|_{\RR}$ is non-decreasing and thus $\int_a^b f=A(f(b))-A(f(a))\ge 0$.
\end{proof}

\begin{Note}{\rm 
The functions in $\mathcal{F}_a$ are much easier to deal with. They can be extended past $+\infty$ essentially by reintepreting their Puiseux series at $+\infty$ as a normal form of a surreal variable. The Puiseux series can then be integrated term by term, resulting in a ``good'' integral. A genetic definition may  then be obtained by minor adaptations of the construction in \cite{FO}. 

For the general family $\mathcal{F}$ we need the machinery of generalized \'Ecalle-Borel summability of transseries \cite{Book, Ecalle1,Ecalle2} and Catalan averages \cite{Menous} to {\em establish the properties of the extensions and integrals} but  not necessarily to calculate them. We will separate the technical constructions accordingly, to enhance readability.
}\end{Note}
\begin{Note}\label{exp}
  \rm Since the surreal exponential $\exp$ (see \S\ref{S2}) is surreal-analytic at any point, it is clear that any of its antiderivatives is of the form  exp+$C$ where $C$ is a constant, at least locally.  We show that  an antidifferentiation operator is defined on a wide class of functions; this antidifferentiation operator gives  $C=0$ everywhere for exp. Via Proposition \ref{existint}, $C=0$ translates into $\int_0^{\omega} e^s ds=e^{\omega}-1$ as expected. We note that this stands in contrast to Norton's aforementioned proposed definition of integration which was shown by Kruskal to integrate $e^s$ over the range $[0,\omega]$ to the wrong value $e^\omega$ \cite[page 228]{CO2}.
\end{Note}
\begin{Note}[\emph{Cautionary Note}]\label{NCautionary}
  {\rm A general $C^\infty$ function $f$  cannot  be (correctly) extended in an infinitesimal neighborhood of a point by its Taylor series. This is the case even if the Taylor series converges--unless, of course, the series converges to $f$, in which case $f$ is  {\em analytic} at that point. An example is $e^{-1/x^2}$ extended by zero at zero. The Taylor series at $0$ is convergent (trivially) since it is the zero series. But $e^{-1/x^2}$ is not $0$ in $\No$ for infinitesimal arguments. This function has a convergent {\em transseries}, also trivially, since $ e^{-1/x^2}$ is its own transseries, and therefore provides a correct extension.} 
\end{Note}

\vspace{0.3cm}
The precise technical setting is as follows.
\begin{Theorem}[Existence of extensions and strong antiderivatives]\label{Tmainpos}
  {\rm On $\mathcal{F}$ there exist
    \begin{enumerate}[(a)]
    \item an extension operator $\mathsf{E}$ from $\mathcal{F}$ to $\mathcal{F^*}=\{f^*:f \in \mathcal{F} \}$ (see \S 1.1) which is linear, multiplicative and preserves exp, log and $x^r$ for $r\in\RR$.

\item \label{itemTA} a \emph{linear}   operator $A:\mathsf{E}(\mathcal{F})\to\mathcal{F^*}$ with the properties
  \begin{enumerate}[1.]
  \item  $[A(f)]'=f$;
\item If $F\in T[\No]$ and $(F|_\RR)'=f \in \mathcal{F}$ exists, then there is a $C\in \No$ such that $A(f)=F+C$.
  \end{enumerate}
\end{enumerate}
}\end{Theorem}

The definitions of $A$ and $\mathsf{E}$ in the just-stated theorem follow Norton's original integral scheme, using inequalities with respect to earlier defined functions and earlier values of the integral. Unlike  Norton's definition which was found to be intensional \cite[page 228]{CO2}, ours are shown to only depend on the values of the functions involved. 

\vspace{0.3cm}
\z 
For instance the definition of Ei is
\begin{equation}\label{DefEi1}
\mathrm{Ei}(x)=\left\{e^x\hspace{-.7 em}\sum_{k\in \ell(x)} \frac{k!}{x^{k+1}}-Cx^{-1/2}, S^L \Bigg|e^x\hspace{-.7em}\sum_{k\in \ell(x)}\frac{k!}{x^{k+1}}+Cx^{-1/2}, S^R\right\}
\end{equation}
Here $\ell(x)$ is the ``least term'' set of indices
\begin{equation}
  \label{eq:defell}
  \ell(x): = \left \{ k\in \mathbb{N}:k\leq   x   \right \}.
\end{equation}
If $x$ is finite, $\ell(x)$ is a finite set and \eqref{DefEi1} equals the least term summation of the series, see Note \ref{Nlts} below. If $x$ is infinite, then clearly $\ell(x)=\NN$. In this case, the sum is to be interpreted as an absolutely convergent series in the sense of Conway (see \S\ref{S2}). The bounds in the sets $S^L$ and $S^R$ are based on local Taylor polynomials, as in \cite{FO}; see \S\ref{S4.3} for details.

\begin{Note}[Least term summation]\label{Nlts}{\rm 
 More specifically, when a series $\sum_{n=0}^\infty c_n z^n$  has zero radius of convergence, then $\limsup_n |c_n z^n|=\infty$, $\forall z\ne 0$.  Say $c_n=n!$ and $z=1/100$. Then,  using Stirling's formula for the factorial, it is easy to see that $n! z^n$ decreases in $n$  until approximately $n=100$, and then increases in $n$. In this case, the least term set of indices $\ell(z)$ is $\{0,1,...,100\}$. The least term is reached approximately at $n=100$ and is of the order $O(\sqrt{z}e^{-1/|z|})$. Summing to the least term means summing the first 100 terms when $z=1/100$, or the first $N$ terms if $z=1/N$. The error obtained in this way is proved in \cite{CK} to be exponentially small, rather than polynomially small as would be the case when summing a fixed number of terms of the  series.}
\end{Note}}
See \S\ref{SSomexam} for further examples.

\section{Surreal numbers: overview and preliminaries}\label{S2}

There are a variety of recursive \cite{CO2}, \cite{EH7}, \cite[also see, \cite{AE1}, \cite{AE2}]{EH1} and non-recursive \cite[page 65; also see, \cite{GO}]{CO2} \cite[page 242]{EH4} constructions of the class $\No$ of surreal numbers, each with their own virtues. For the sake of brevity, here we adopt Conway's construction based on sign-seqences \cite[page 65]{CO2}, which has been made popular by Gonshor \cite{GO}.

In accordance with this approach, a {\it surreal number\/} is 
a function $a : \lambda \to \{-,+\}$ where $\lambda$ is an ordinal called the \emph{length} of $a$.
The class $\No$ of surreal numbers so defined carries a canonical linear ordering: $a< b$ if and only if $a$ is
(\emph{lexicographically}) less than $b$ with respect to the linear ordering on
 $\{-,+\}$, it being understood that $- < \text{\emph{undefined}}\ < +$. 
As in \cite{EH5}, we define the canonical partial
ordering $<_s$ on $\No$ by: $a<_s b$ (``$a$ is \emph{simpler than} $b$") if and only if $a$
is a proper initial segment of $b$. 

A tree $\left\langle {A,{ < _A}} \right\rangle $ is a partially ordered class such that for each $x \in A$, the class 
$\left\{ {y \in A: y { < _A} x} \right\}$ of \emph{predecessors} of $x$ is a set well ordered by ${ < _A}$. If each member of  $A$ has two immediate successors and every chain in $A$ of limit length (including the empty chain) has 
one immediate successor, the tree is said to be a \emph{full binary tree}. Since a full binary tree has a level for each ordinal, the
universe of a full binary tree is a proper class.

\begin{Proposition}\label{sur1}{\rm

 $\left\langle {\bf No \, <, <_s} \right\rangle$ is a lexicographically ordered full binary tree.

}
\end{Proposition}

Central to the algebraico-tree-theoretic development of the theory of surreal numbers is the following consequence of Proposition \ref{sur1}.

\begin{Proposition}\label{sur2}{\rm

 If  $L$ and $R$ are (possibly empty) sub\emph{sets} of ${\bf No}$ for which every member of $L$ precedes every member of $R$ (written $L<R$), there is a \emph{simplest} member of ${\bf No}$ lying between the members of $L$  and the members of  $R$ \cite[pages 1234-1235]{EH5} and \cite[ Proposition 2.4]{EH7}. 
 }
\end{Proposition}
 
 Co-opting notation introduced by Conway, the simplest member of  ${\bf No}$ lying between the members of  $L$ and the members of $R$  is denoted by the expression $$\{L|R\}.$$ 

Following Conway \cite[page 4]{CO2}, if $x=\left\{ L|R\right\}$, we write $x^{L}$ for the
\emph{typical} member of $L$ and $x^{R}$ for the typical member of $R$; for $x$ itself, we write $\left\{ x^{L}|x^{R}\right\} $; and $x=\left\{
a,b,c,...|d,e,f,...\right\} $ means that $x=\left\{ L|R\right\} $ where $%
a,b,c,...$ are the typical members of $L$ and $d,e,f,...$ are the typical
members of $R$. In accordance with these conventions, if $L$ or $R$ is empty, reference to the corresponding typical members may be deleted. So, for example, in place of $0=\left \{ \varnothing |\varnothing \right \}$, one may write $0=\left \{ | \right \}$.     

Each  $x\in\bf No$ has a \emph{canonical representation} as the simplest member of ${\bf No}$ lying between its predecessors on the left and its predecessors on the right, i.e.$$x=\{L_{s\left ( x \right )}| R_{s\left ( x \right )}\},$$ where  $L_{s\left( x \right)}=\left\{ {a \in {\bf No}:a <_s x\;{\rm{and}}\;a < x} \right\}$ and $R_{s\left( x \right)}=\left\{ {a \in {\bf No}:a <_s x\;{\rm{and}}\;x < a} \right\}$. 

By now letting $x=\{L_{s\left ( x \right )}| R_{s\left ( x \right )}\}$ and $y=\{L_{s\left ( y \right )}| R_{s\left ( y \right )}\}$, $+ , - $ and $\cdot$ are defined by recursion for all $x,y\in{\bf No}$ as follows, where  the typical members $x^L $,  $x^R$,  $y^L $ and  $y^R $ are understood to range over the members of  $L_{s\left( x \right)},R_{s\left( x \right)} ,L_{s\left( y \right)} $ and  $R_{s\left( y \right)} $, respectively.

\bigskip 
\emph{Definition of}  $x + y.$
 $$x + y = \left\{ {x^L  + y,x + y^L |x^R  + y,x + y^R } \right\}.$$
 
 \emph{Definition of}  $- x.$
 $$- x = \left\{ {- x^R | - x^L } \right\}.$$
 
 \emph{Definition of}  $xy.$
 
\begin{eqnarray*}
\!\!\!xy &=&\{x^{L}y+xy^{L}-x^{L}y^{L},x^{R}y+xy^{R}-x^{R}y^{R}| \\
&&\quad \quad \qquad \qquad \qquad \text{\quad }
x^{L}y+xy^{R}-x^{L}y^{R},x^{R}y+xy^{L}-x^{R}y^{L}\}\text{.}
\end{eqnarray*}

\medskip

Despite their cryptic appearance, the definitions of sums and products on $\No$ 
have natural interpretations that essentially assert that the sums and products of elements of $\No$ are the simplest possible elements of $\No$ consistent with $\No$'s structure as an ordered field \cite[page 1236]{EH4},  \cite[pages 252-253]{EH5}. The constraint on additive inverses, which is a consequence of the definition
of addition \cite[page 1237]{EH5}, ensures that the portion of the surreal
number tree less than $0$ is (in absolute value) a mirror image of the portion
of the surreal number tree greater than $0$, $0$ being the simplest element of the surreal number tree.

A subclass $A$ of $\No$ is said to be \emph{initial} if $b\in A$ whenever $a\in A$  and $b<_s a$. Although there are many isomorphic copies of the order field of reals in $\No$, only one is initial. This ordered field, which we denote $\RR$, plays the role of the reals in $\No$. Similarly, while there are many well-ordered proper subclasses of $\No$ in which $x<y$ if and only if $x<_s y$, only one is initial. The latter, which consists of the outermost right branch of $\left\langle {\bf No \, <, <_s} \right\rangle$, is identified as $\No$'s ordered class $On$ of ordinals. See Figure 1.

A subclass $B$ of an ordered class $\left \langle A,< \right \rangle$  is said to be \emph{convex}, if $z\in B$ whenever $x,y\in B$ and $x < z < y$. 

The following are consequence $\bf{No}$'s structure as a lexicographically ordered binary tree.

\begin{Proposition}[\cite{EH5}: Theorems 1 and 2]\label{sur3} {\rm

(i) Every non-empty convex subclass of $\No$ has a simplest member; (ii) if $x\in \bf No$ and $L,R$ are a pair of subsets of $\No$ for which $L<R$, then $x = \left\{ {L|R} \right\}$ if and only if $L<\left\{ x\right\} <R$
and $\left\{ a\in {\bf No}:L<\left\{ a\right\} <R\right\} {\it \subseteq }\left\{
a\in {\bf No}:L_{s(x)}<\left\{ a\right\} <R_{s(x)}\right\}$.

}
\end{Proposition}

The non-zero elements of an ordered group can be partitioned into equivalence classes consisting of all members of the group that mutually satisfy the Archimedean condition. If $a$ and $b$ are members of distinct Archimedean classes and $\left |a \right |< \left |b \right |$, then we write $a \ll b$ and $a$ is said to be \emph{infinitesimal (in absolute value) relative to} $b$.

An element of  ${\bf No}$ is said to be a \emph{leader} if it is the simplest member of the positive elements of an Archimedean class of  ${\bf No}$. Since the class of positive elements of an Archimedean class of  ${\bf No}$ is convex, the concept of a leader is well defined. There is a unique mapping--\emph{the} $\omega$\emph{-map}--from ${\bf No}$ onto the ordered class of leaders that preserves both $<$ and $<_s$. The image of $y$ under the $\omega$-map is denoted $\omega^y$, and in virtue of its order preserving nature, we have: for all $x,y\in \bf {No}$, $$ \omega^x \ll \omega^y \;\text{if and only if} \; x<y.$$

\noindent Using the $\omega$-map along with other aspects of ${\bf No}$'s $s$-hierarchical structure and its structure as a vector space over $\RR$, every surreal number can be assigned a canonical ``proper name'' or \emph{normal form} that is a reflection of its characteristic
$s$--hierarchical properties. These normal forms are expressed as
sums of the form  $$\sum\limits_{\alpha  < \beta } {\omega ^{y_\alpha  } .r_\alpha} $$
where  $\beta $ is an ordinal,  $\left( {y_\alpha} \right)_{\alpha  < \beta } $ is a strictly decreasing
sequence of surreals, and  $\left( {r_\alpha  } \right)_{\alpha  < \beta } $ is a sequence of non-zero
real numbers. Every such expression is in fact the normal
form of some surreal number, the normal form of an ordinal being just its \emph{Cantor normal form} \cite[pages 31-33]{CO2}, \cite[\S3.1 and \S5]{EH5}, \cite{EH6}.

Making use of these normal forms, Figure ~\ref{fig:first} offers a glimpse of the some of the early stages of the recursive unfolding of  ${\bf No}$.

\begin{figure}[ht!]

\centering\includegraphics[width=0.95\textwidth]{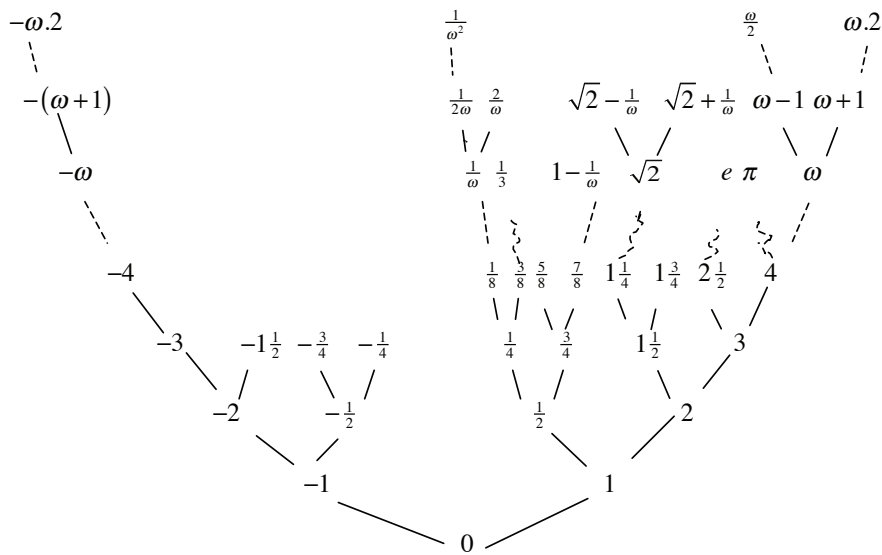}

\caption{Early stages of the recursive unfolding of  ${\bf No}$}
\label{fig:first} 
\end{figure}

When surreal numbers are represented by their normal forms, the sums, products and order on $\No$ assume the following more tractable termwise and lexicographical forms, 
where ``dummy" terms with zeros for coefficients are understood to be inserted and
deleted as needed. 

\begin{Proposition}[\cite{EH5}, Theorem 16; also see, \cite{GO}, pages 67-70 and \cite{AL}, pages 235-246]\label{sur4}

$$\sum\limits_{y\in \/\bf{No}}\omega ^{y}.a_{y}+\sum\limits_{y\in \/\bf{No}}\omega
^{y}.b_{y}=\sum\limits_{y\in \/\bf{No}}\omega ^{y}.\left( a_{y}+b_{y}\right),$$ 

$$\sum\limits_{y\in \/\bf{No}}\omega ^{y}.a_{y}\cdot\sum\limits_{y\in \/\bf{No}}\omega
^{y}.b_{y}=\sum\limits_{y\in \/\bf{No}}\omega ^{y}.\left[ \sum_{\stackrel{(\mu,\nu) \in \bf{No}\times \bf{No}}{{\mu}+{\nu} =y}} a_{\mu }b_{\nu }\right]
\,, $$

$$\quad \sum\limits_{y\in \/\bf{No}}\omega ^{y}.a_{y}<\sum\limits_{y\in
\bf{No}}\omega ^{y}.b_{y},{\it if}{\it \ }a_{y}=b_{y} \text{\it\ for all} {\it \ } y>
{\it some}{\it \ }x\in \bf{No} {\it \  and} {\it \ a_{x}<b_{x}}.$$

\end{Proposition}

The following result is an immediate consequence of Conway's definition of normal forms and (\cite[pages 32-33]{CO2} and \cite[page 1247]{EH5}) their lexicographical ordering.

\begin{Proposition}\label{surtrunk} {\rm For all $x\in \bf {No}$ and all positive $y\in \bf {No}$, $x = \left\{ {x-y|x+y} \right\}$, wherever all of the exponents in the normal form of $x$ are greater than all of the exponents in the normal form of $y$.

} \end{Proposition}

Since every ordered field $A$ contains a unique isomorphic copy, $\mathbb{Q}_A$, of the ordered field of rational numbers, $a\in A$  may be said to be \emph{infinite} (\emph{infinitesimal}) if $\left | a \right |$  is greater than (less than) every positive member of $\mathbb{Q}_A$. Thus, in virtue of the lexicographical ordering on normal forms, a surreal number is infinite (infinitesimal) just in case the greatest (i.e. the zeroth) exponent in its normal is greater than (less than) $0$. As such, each surreal number $x$ has a canonical decomposition into its \emph{purely infinite part} $\Pi(x)$, its \emph{real part} re$(x)$, and its \emph{infinitesimal part} $\amalg(x)$, consisting of the portions of its normal form all of whose exponents are $>0$, $=0$, and $<0$, respectively. 

There is a notion of convergence in ${\bf No}$ for sequences and series of surreals that can be conveniently expressed using normal forms written as above with dummy terms.  Let $x\in\No$ and for each $y\in\No$, let $r_{y}(x)$ be the coefficient of $\omega^y$ in the normal form of $x$, it being understood that $r_{y}(x)=0$, if $\omega^y$ does not occur.  Also let
$\left \{ x_{n}: n\in \mathbb{N}\right \}$ be a sequence of surreals written in normal form. Following Siegel \cite[page 432]{SI}, we write \[x=\lim _{n\rightarrow \infty }x_{n}\]

\noindent to mean \[r_y\left ( x \right )=\lim _{n\rightarrow \infty }r_y \left ( x_{n} \right ),  \:   \text{for all}\ y\in\No,\]

\noindent and say that $x_n$ \emph{converges} to $x$. We also write \[x=\sum_{n=0}^{\infty }x_n\]

\noindent to mean the partial sums of the series converge to $x$.

Among the convergent sequences and series of surreals are those whose mode of convergence is quite distinctive. In particular, for each $y\in\bf {No}$, there is an $m\in\NN$ such that $r_{y}(x_n)=r_{y}(x_m)$ for all $n\geq m$. Thus, for each $y\in\bf {No}$,
 \[r_y\left ( x \right )=\lim _{n\rightarrow \infty }r_y \left ( x_{n} \right )=r_{y}(x_m),\]

\noindent where $m$ depends on $y$.
Following Conway, we call this mode of convergence \emph{absolute convergence}. Moreover, we will call the normal form to which an absolutely convergent series $x_n$ of normal forms converges the \emph{Limit} of the series and write  ${\rm Lim}_{n\to \infty}x_n$. We use ``Limit" as opposed to ``limit" to distinguish the surreal notion from its classical counterpart.

Relying on the above and classical combinatorial results of Neumann (\cite[pages 206-209]{N},\cite[Lemma 3.2]{SI}, \cite[pages 260-266]{AL}), one may prove \cite[pages 432-434]{SI} the following theorem of Conway \cite[page 40]{CO2}, which is a straightforward application to $\No$ of a classical result of Neumann \cite[page 210]{N}, \cite[page 267]{AL}.

\begin{Proposition}\label{sur5}{\rm Let $f$ be a formal power series with real coefficients, i.e. let \[f\left ( x \right )=\sum_{n=0}^{\infty }x^{n}r_{n}.\]

\noindent Then $f\left ( \zeta \right )$ is absolutely convergent for all infinitesimals $\zeta$ in $\No$. 
}
 \end{Proposition}

Conway's theorem also has the following multivariate formulation \cite[page 435]{SI}.

\begin{Proposition}\label{sur6}{\rm  Let $f$ be a formal power series in $k$ variables with real coefficients, i.e. let \[f\left (x_1,...,x_k \right )\in\RR[[x_1,...,x_k]].\]
\noindent Then $f\left ( \zeta_1,..., \zeta_k \right )$ is absolutely convergent for every choice of infinitesimals $\zeta_1,..., \zeta_k$ in $\No$.} 
 \end{Proposition}
 
 This can also be written in the following useful form.
  
\begin{Proposition}\label{sur7}{\rm Let $\{c_{\mathbf{k}}: \mathbf{k}\in\NN^m\}$ be any multisequence  of real numbers and $h_1,...,h_m$ be  infinitesimals. Also let $\mathbf{h}^{\mathbf{k}}=h_1^{k_1}\cdots h_m^{k_m}.$ Then
  \begin{equation}
    \label{eq:sumc}
    \sum_{|\mathbf{k}|\ge 0} c_{\mathbf{k}} \mathbf{h}^{\mathbf{k}}
  \end{equation}
is a well-defined element of $\No$. 
}\end{Proposition}

The following result, in which $x_n$ and $y_n$ are absolutely convergent series of normal forms, collects together some elementary properties of absolute convergence in $\No$. Many are very similar to the properties of the usual limits. 

\begin{Proposition}\label{Pcop}{\rm 
Let ${\rm Lim}_{n\to \infty}x_n=x$ and ${\rm Lim}_{n\to \infty}y_n=y$, and further let $h\ll 1$, $\tau > 0$ and $A,B\in\No$. Then 
\begin{align}
  \label{eq:convprop}
(a)\ &   {\rm Lim}_{n\to \infty}(A x_n+B y_n)=A x+B y       ;\nonumber\\(b)\ & {\rm Lim}_{n\to \infty}x_ny_n=xy ;\nonumber\\ (c)\ & x\ne 0\Rightarrow {\rm Lim}_{n\to \infty} \frac{1}{x_n}=\frac{1}{x};\\ (d)\   & (\exists K)(\forall n)(|x_n|<K);\nonumber\\ (e)\ & {\rm Lim}_{n\to \infty}h^n=0 ;\nonumber\\ (f) \ & (\forall n)(|x_n|\le \tau)\Rightarrow |x|\le \tau.\nonumber
\end{align}
   } \end{Proposition}

  \begin{proof}[Proof of Proposition \ref{Pcop}]
  
  (a) and (b) are proved in \cite[page 271]{AL}, (d) is evident since no set is cofinal with $\bf No$, (e) follows from Proposition \ref{sur2} and (f) follows from (e).
  For (c), since ${\rm Lim}_{n\to \infty}x_n\neq0$, there is a greatest $y\in\bf{No}$ such that $r_y(x_n)$ is not eventually zero. Thus, for sufficiently large $n$, $x_n=r_y\omega^{y}(1+h_n)$, where $h_n$ is infinitesimal, and, so, it suffices to establish the result for $x_n$ of the form $1/(1+h_n)$. Since $1/(1+h_n)-1=-h_n(1+h_n)^{-1}$ and ${\rm Lim}_{n\to \infty}h^n=0$ the coefficients of leaders in $h_n$ eventually vanish, and, as such, eventually vanish for $-h_n(1+h_n)^{-1}$.
  \end{proof}

$\mathbf{No}$ admits an inductively defined exponential function $\exp :\mathbf{No}\rightarrow \mathbf{No}$ together with a natural interpretation of real analytic functions restricted to the finite (i.e. non-infinite) surreals that makes it a model of the theory of real numbers endowed with the exponential function $e^x$ and all real analytic functions restricted to a compact box \cite{vdDE}. The exponential function, which was introduced by Kruskal, is developed in detail by Gonshor in \cite[Chapter 10]{GO}. Norton and Kruskal independently provided inductive definitions of the inverse function $\ln$, but thus far only an inductive definition of $\ln$ for surreals of the form $\omega^{y}$ has appeared in print; see \cite[page 161]{GO} and below. Nevertheless, since each positive surreal $x$, written in normal form, has a unique decomposition of the form $$x=\omega^{y}r(1+\epsilon),$$ where $\omega^{y}$ is a leader, $r$ is a positive member of $\RR$ and $\epsilon$ is an infinitesimal, $\ln(x)$ may be obtained for an arbitrary positive surreal $x$ from the equation $$\ln(x)=\ln(\omega^{y})+\ln(r)+\ln(1+\epsilon),$$
where  $\ln(\omega^{y})$ is inductively defined by 
$$\left \{ \ln \left ( \omega
^{y^{L}} \right )+n, \ln \left ( \omega ^{y^{R}} \right )-\omega ^{\left
( y^{^{R}}-y \right )/n}\Big| \ln \left ( \omega ^{y^{R}} \right )-n,
\ln \left ( \omega ^{y^{L}} \right )+\omega ^{\left (y- y^{^{L}} \right
)/n} \ \right\},$$

\noindent and

$$\ln(1+\epsilon)=\sum_{k=1}^{\infty } \frac{\left ( -1 \right )^{k-1} \epsilon^{k}}{k}.$$

\noindent Moreover, since $\ln$ is analytic, a genetic definition of $\ln(1+z)$ for infinitesimal values of the variable can be provided in the manner discussed in detail in the work of Fornasiero \cite[Pages 72-74]{FO}, and sketched below in Definition \ref{Def.6}.

Readers seeking additional background in the theory of surreal numbers may consult  \cite{AL}, \cite{CO2}, \cite{EH5}, \cite{EH7}, \cite{GO}, \cite{SCST} and \cite{SI}.

\vspace{36pt}

\section{Transseries.  \'Ecalle-Borel (EB)  transseriable functions. }\label{Sec5.3}
This section reviews some classical results in EB summability theory.  Except for the conventions and notation  it  can by skipped by readers familiar with these notions. For more details on transseries see \cite{Book,Edgar,CostinT,Joris,ADH}.

{\em Borel and EB summability.} EB summability applies to series of the form
\begin{equation}
  \label{eq:bsumbeta}
 \tilde{f}:=\sum_{k=-M}^\infty c_k x^{-(k+1)\beta},\ \Re\beta>0 .
\end{equation}
Since the sum from $-M$  to $ 0$  is finite we can assume without loss of generality that $M=0$. Typically, the definitions assume that $\beta=1$, and as we will see there is essentially  no  loss of generality in so doing. 
\subsection{Classical Borel summation of series}
\begin{Definition}\label{DefBE}
  {\rm The \emph{Borel sum} of a formal series $\tilde{f}$, denoted $f=\mathcal{LB}\tilde{f}$ (where $\mathcal{L}$ is the Laplace transform and $\mathcal{B}$ the Borel transform),  exists when steps (ii) and (iii) in the following process can be carried out. 

(i) Take the Borel transform $\tilde{F}=\mathcal{B}\tilde{f}$ of $\tilde{f}$. ($\tilde{F}$ is still a formal power series defined as the term-by-term inverse Laplace transform of $\tilde{f}$: $\mathcal{B}\sum_{k=0}^\infty c_{k} x^{-(k+1)\beta}=\sum_{k=0}^\infty c_k p^{-(k+1)\beta-1}/\Gamma((k+1)\beta)$.)

(ii) Assuming $\mathcal{B}\tilde{F}$ converges to $F$,  analytically continue $F$ on $\RR^+$ assuming this is  possible.

(iii) Take the Laplace transform, $f=\mathcal{L}F$, provided exponential bounds exist for $F$, say, if $\exists \nu>0\text{ such that}\; \sup_{x> \nu}|e^{-\nu x} F(x)|<\infty$.

For example, $$\mathcal{LB}\sum_{k=0}^\infty k!(-x)^{-k-1}=\mathcal{L}(1+p)^{-1}=-e^x\text{Ei}(-x).$$ For further details on Borel sums, see \cite{Book}.
}\end{Definition}
If $\beta=1$, then $\mathcal{B}\tilde{f}$ is analytic at $p=0$; otherwise it is ramified-analytic, $\mathcal{B}\tilde{f}=1/pA(p^\beta)$ where $A$ is analytic. Given this trivial transformation,  we will simply take $\beta=1$.

{\em Analyticity.} Standard complex analysis arguments (e.g. combining Morera's theorem with Fubini) show  that $f=\mathcal{LB}\tilde{f}$ is real analytic for large $x\in\RR\cap (\nu,\infty)$.
 \begin{Definition}\label{DGevrey}
  {\rm The power series
$$\sum_{k=-M}^{\infty}\frac{c_k}{x^{k+1}}$$
is Gevrey-one if there are $C,\rho >0$ such that for all $k$, $|c_k|\le k!C \rho^{-k}$. 
}\end{Definition}
\begin{Note}\label{SumLT}
  {\rm 
It is  known \cite[pages 104--109]{Book} that Borel summable {\em formal series} form a differential field, isomorphic to the field of Borel summed series. It is also known \cite[Theorem 2]{CK}\footnote{We emphasize that the proofs in \cite{CK} do not use the origin of the transseries.} that Borel summable functions, as well as EB-summable functions originating in generic linear or nonlinear systems of ODEs or difference equations, have the property that the difference between the function $f$ and its asymptotic series truncated to its least term $c_k x^{-k}$ ($k$ is dependent on $x$, and is roughly $k=\lfloor \rho x\rfloor$) 
satisfies
\begin{equation}
  \label{eq:SLT}
\left|f(x)-\sum_{k\in\ell(x)}\frac{c_k}{x^k}\right|   \le C e^{-|\rho x|}x^b
\end{equation}
where $C,\rho,b$ are  constants that can be estimated relatively easily in concrete examples. 

Inequality \eqref{eq:SLT}, when the constants are sharp,   is  the summation to the least term estimate; see Note \ref{Nlts}. Summation to the least term,  introduced by Cauchy and Stokes, was further developed by Berry \cite{Berry1}, and later extended by Berry, Delabaere, Howls,  Olde Daalhuis and others; for references, see \cite{Olde1}. Here, however, we only need the classical notion.}
\end{Note} 
\subsection{Transseries} Somewhat informally,  transseries were already used in the late 19th century since they arise naturally in the study of differential equations.  Let's start with the simplest nontrivial differential equations in a neighborhood of $x=0$:
  \begin{equation}\label{threetypes}
    a_i(x)y''+(x^2-x)y(x)+y(x)=0;\ \ a_1(x)=1, \ a_2(x)=x^2,\
 a_3(x)=x^3.
  \end{equation}
To study the properties of \eqref{threetypes} it is useful to first divide the equation by $a_i$. When $i=1$, the resulting equation has analytic coefficients, and by the general theory of ODEs, it has an  analytic fundamental system of solutions at zero (which are entire since the equation has no singularities in $\CC$). In particular, it can be solved by a convergent power series of the form,
\begin{equation}
  \label{eq:sl1}
  y=A\(1-\tfrac12 z^2-\tfrac1{24}z^4+\cdots\)+B\(z-\tfrac1{12}z^4-\tfrac1{120}z^6+\cdots\).
\end{equation}
When $i=2$, $z=0$  is a point where the coefficients are meromorphic, but not analytic. The order of the pole is sufficiently low  and the singular point is {\em regular} in the sense of Frobenius. This being the case, there still exists a fundamental set of solutions as convergent  series, though the powers are not necessarily integer anymore and, non-generically, logs may get mixed in. In our example, we obtain
\begin{equation}
  \label{eq:case2}
   y=A\sqrt{z}\(1-\tfrac12 z-\tfrac1{8}z^2+\cdots\)+B\(z-\tfrac1{12}z^4-\tfrac1{120}z^6+\cdots\).
\end{equation}
The situation changes abruptly when the order of the poles exceeds the order of the equation, that is, when $i=3$. The singular point becomes {\em irregular}. Now the space of formal power series solutions is just {\em one-dimensional}, that is
  \begin{equation}
    \label{case3}
    y=A\sum_{k=0}^{\infty}k!z^{k+1}
  \end{equation}
whereas a second order ODE must have a two-dimensional space of solutions; furthermore,  series \eqref{case3}  is divergent. The second family of solutions is {\em not} a power series at zero but rather
 \begin{equation}
    \label{case3b}
    y=Be^{-1/z}.
  \end{equation}
The general solution is thus
\begin{equation}
    \label{case31}
    y=A\sum_{k=0}^{\infty}k!z^{k+1}+Be^{-1/z}.
  \end{equation}
However, since the power series in \eqref{case3}  has  radius of convergence zero, $y$ in \eqref{case31}  --perhaps the  simplest nontrivial transseries at zero-- is now only a {\em formal solution}.  

Imagine now that we have a singular nonlinear ODE. One of the simplest irregularly singular  nonlinear ODE is obtained by the change of dependent variable  $h(z)=1/y(z)$ in \eqref{threetypes}. Clearly, the general formal solution for the  just-said equation is
\begin{equation}
  \label{eq:nonl1}
  \frac{D}{\sum_{k=0}^{\infty}k!z^{k+1}+Ce^{-1/z}}=:\frac{D}{y_1+C_2 e^{-1/z}}.
\end{equation}
If $z>0$, then   $ e^{-1/z}\ll z$, and hence $e^{-1/z}\ll y_1$. We let $y_2=y_1(z)/z$ where now $y_2=1+O(z)$, and further let $\tilde{y}=1/y_2$, which can be re-expanded as a power series at zero;   moreover, we can further expand, simply using the geometric series,
\begin{multline}
  \label{eq:nonl1}
  \frac{D}{y_1+Ce^{-1/z}}=:\frac{D\tilde{y}}{z(1+C_2 e^{-1/z}\tilde{y})}\\=\frac{D}{z}\left(\tilde{y}-C_2\tilde{y}^2e^{-1/z}+C_2^2\tilde{y}^3e^{-2/z}+\cdots\right)=
z^{-1}\sum_{k=0}^{\infty}\tilde{y}_j C^j e^{-j/z}
\end{multline}
where $\tilde{y}_j$ are formal power series. This is a more general {\em level one transseries at $0^+$}.  
\begin{Definition}[Informal definition]
  A transseries with generators $\mu_1,\mu_2,...,\mu_n$ considered as ``variables $\ll 1$'' is a formal sum 
  \begin{equation}
    \sum_{k_1,k_2,...,k_n>-M}c_{k_1,k_2,...,k_n}\mu_1^{k_1}\mu_2^{k_2}\cdots \mu_n^{k_n}=:\sum_{\bfk>-M}c_{\bfk}\boldsymbol{\mu}^{\bfk}.
  \end{equation}
\end{Definition}
\begin{Note}
  {\rm It is customary to work with the case when the variable tends to $\infty$ rather than to 0. This is achieved simply by taking $x=1/z$. One reason for adopting this convention is that equations arising in practice most often have their worst singularities at $\infty$. Other reasons relate to  algebraic simplicity. For instance, repeated differentiation of $e^{-1/z}$ obviously leads  to more complicated expressions than the repeated differentiation of $e^{-x}$.}
\end{Note}
\z {\bf Example.}: Note that \eqref{eq:nonl1} is a transseries with generators $\mu_1=z$ and $\mu_2=e^{-1/z}$. 
Also note that \eqref{trans2} has generators $\mu_1=1/x,...,\mu_j=x^{\beta_j}e^{-\lambda_j x}$.

The survey paper \cite{Edgar} is an excellent self-contained introduction to transseries and their topology, with connections to normal forms of surreal numbers, as well as ample references.
\begin{Definition}\label{defconv}{\rm 
  The \emph{transseries topology} (see \cite{Book,Edgar}) is defined by the following convergence notion. Let  $\sum_{\bfk>-M}c_{\bfk}^{[m]}\boldsymbol{\mu}^{\bfk}$ be a sequence of transseries, where the superscript $[m]$ designates the $m$th element of the sequence and $c_{\bfk}^{[m]}$ designates the sequences of coefficients of the $m$th element. Then,  
  \begin{equation}
    \label{eq:defc}
    \lim_{m\to\infty} \sum_{\bfk>-M}c_{\bfk}^{[m]}\boldsymbol{\mu}^{\bfk}= \sum_{\bfk>-M}c_{\bfk}\boldsymbol{\mu}^{\bfk} \text{ if and only if } \forall\bfk\exists M \text{ such that }\forall m>M, c_{\bfk}^{[m]}=c_{\bfk},
  \end{equation}
i.e., {\em if and only if all the coefficients eventually become  those of the limit transseries (rather than merely converge to them).}
}\end{Definition}
  \begin{Definition} [Borel summable transseries]\label{DEBstr}
{\rm 
 Rewriting \eqref{trans2} as 
\begin{equation}\label{trans21}
  \tilde{T}=  \sum_{\bfk\ge \bfk_0,l\in\NN}c_{\bfk,l}x^{\boldsymbol{\beta}\cdot \bfk}e^{-\bfk \cdot \boldsymbol{\lambda} x}x^{-l}= \sum_{\bfk\ge \bfk_0}x^{\boldsymbol{\beta}\cdot \bfk+1}e^{-\bfk \cdot \boldsymbol{\lambda} x}\tilde{y}_{\bfk}(x)
  \end{equation}
where $\tilde{y}_{\bfk}(x)=\sum_{l=0}^\infty c_{\bfk,l}x^{-l-1}$ are formal power series, we say that the transseries $\tilde{T}$ is \emph{Borel-summable} if there is a $\bfk$-independent exponential bound (the same constant  $\nu$ in Definition \ref{DefBE} for all $\bfk$)  such that all the $\tilde{y}_{\bfk}(x)$ satisfy  Definition \ref{DefBE} and there are $A,B, r>0$ such that for all $\Re x>r$ we have $\mathcal{LB} \tilde{y}_{\bfk}(x)\le A B^k$. 
When the conditions above are met,  $\tilde{T}$ is called a  \emph{Borel summable transseries}, and $\mathcal{LB}\tilde{T}$ is its Borel sum.
}\end{Definition}
 These estimates imply
 \begin{Lemma}\label{LEBstr}{\rm 
 \begin{equation}
  \label{eq:BsumT}
  T:=\mathcal{LB}\tilde{T}=\sum_{\bfk\ge \bfk_0}x^{\boldsymbol{\beta}\cdot \bfk+1}e^{-\bfk \cdot \boldsymbol{\lambda} x}\mathcal{LB}\tilde{y}_{\bfk}=\sum_{\bfk\ge \bfk_0}x^{\boldsymbol{\beta}\cdot \bfk+1}e^{-\bfk \cdot \boldsymbol{\lambda} x}y_{\bfk} 
\end{equation}
is a  function series converging to an  analytic function for   $|x^{\boldsymbol{\beta}\cdot \bfk}e^{-\bfk \cdot \boldsymbol{\lambda} x}|<1/B$.
 }\end{Lemma}

\subsection{EB summation of series}
Even in simple cases classical Borel summation fails because of singularities. For instance $\mathcal{LB}\sum_{k=0}^\infty k!x^{-k-1}=\mathcal{L}(1-p)^{-1}$, which is undefined as presented.

\'Ecalle introduced   significant improvements over Borel summation. Among them are the concepts of \emph{critical times} and \emph{acceleration/deceleration}  to deal with mixed powers of the factorial divergence. Last but not least, and the only additional ingredient we will need, is that of \emph{averaging}.

In linear problems, to avoid the singularities of the integrand  on $\RR^+$ (when present),  one can take the half-half average of the Laplace transforms above and below $\RR^+$. In nonlinear problems, on the other hand, the average  of two solutions is not a solution. Though \'Ecalle found constructive, universal  and simple looking averages, notably the Catalan averages,  which successfully replace the naive half-half averages mentioned above, it is altogether nontrivial to construct them and show that they work \cite{Menous}. Invoking Borel transforms followed by analytic continuation along paths avoiding singularities, followed by taking the Catalan averages of these continuations and finally applying the Laplace transforms yields a differential field algebra \cite{Menous}.

\'Ecalle introduced  other important improvements in asymptotics, such as general transseries, extending EB summation to general transseries, and \'Ecalle cohesive continuation, which allows for a good continuation through  ``natural natural'' boundaries. 
\begin{Proposition}\label{consist}{\rm 
Borel summation and more generally   EB summation (see Definition \ref{DEB}) is a differential algebra isomorphism between level one transseries arising in  nonlinear systems of ODEs under the assumptions of \cite{IMRN} and actual functions in $\CC$ or $\RR$. 
}\end{Proposition}
\begin{proof}[Proof of Proposition \ref{consist}]
	
	The Catalan average, \cite[\S IV 2, page 604]{Menous}, an average in the Borel plane (the space obtained after taking the Borel transform)  commutes with convolution and preserves lateral growth (the exponential bounds needed to take Laplace transforms),   entailing the commutation of EB summation of  a formal series with {\bf multiplication}. The isomorphism between EB summable series and functions thus follows directly from \cite{Menous}. For transseries, we note that the Catalan average and the balanced average used in ODEs (cf. \cite{Duke}) coincide. Indeed, the first pair of Catalan weights is simply $\frac12,\frac12$, the same as those of balanced averages, and the first pair  $\frac12,\frac12$ determines  all the other weights of a balanced average (cf. \cite{Duke}). The isomorphism between  balanced-Borel summable transseries and their sums is proved in \cite{Duke}, under assumptions more general than those in \cite{IMRN}.

\end{proof}

\begin{Definition}\label{DEB}{\rm 
  By  abuse of language, \emph{a Borel summable function} is understood to be a function that is equal to the Borel sum of its own asymptotic series. Following \'Ecalle, functions that are EB sums of their own asymptotic series or transseries are called \emph{analyzable}. Borel summable functions are in particular analyzable.
}\end{Definition}

\subsection{General conditions for level-one EB summable transseries in ODEs}\label{SDuke}
To apply the same procedure for general solutions of nonlinear ODEs at irregular singular points, we rely on the transseries obtained in \cite{IMRN,Duke}.

\z As is well known,  by relatively simple algebraic transformations, a higher order
differential equation can be turned into  a first order vectorial
equation (differential system) and vice versa \cite{CL}. The vectorial
form has some technical advantages.

Furthermore, as discussed in \cite{Book}, the equations of classical functions, Painlev\'e equations and others are amenable to the form:                
\begin{eqnarray}
 \label{eqor10}
  \bfy'=\mathbf{f}(x,\bfy)  \qquad \bfy\in\CC^n              
   \end{eqnarray}

\z under the following {\em assumptions:}

(a1) The function $\mathbf{f}$ is analytic
at $(\infty,0)$;

(a2) A condition of nonresonance holds: the numbers
 $\arg\, \lambda_j,\,\,j=1,...,n$ are {\em distinct}, where $\lambda_j$, all  of which are nonzero,  are
the eigenvalues of the linearization

\begin{eqnarray}
  \label{linearized}
  \hat{\Lambda}:=-\left(\frac{\partial f_i}{\partial
    y_j}(\infty,0)\right)_{i,j=1,2,\ldots n}.
\end{eqnarray}

We impose, in fact,  a stronger condition, namely the condition in the paragraph containing \eqref{eq:sngset} below.
Writing out explicitly a few terms in the expansion of $\bf f$,
relevant to leading order asymptotics, we get

\begin{eqnarray}\label{eqor}
{\bf y}'={\bf f}_0(x)-\hat\Lambda {\bf y}+
\frac{1}{x}\hat A {\bf y}+{\bf g}(x,{\bf y})
\end{eqnarray} 
where $\bf g$ is analytic at $(\infty,\mathbf 0)$ and ${\bf g}(x,{\bf y})=
O(x^{-2}, |\mathbf y|^2, x^{-2}\mathbf y)$.

\z {\em Nonresonance}.
For any $\theta>0$, denote by $\mathbb{H}_\theta$ the open half-plane centered on
$e^{i\theta}$.  Consider the eigenvalues contained in  $\mathbb{H}_\theta$,
written as a vector $\tilde{\boldsymbol{\lambda}}=(\lambda_{i_1},...,\lambda_{i_{n_1}})$; 
for these $\lambda$s we have $\arg\lambda_{i_j}-\theta\in (-\pi/2,\pi/2)$.

We require that for all $\theta$, the numbers in the finite set
\begin{equation}
  \label{eq:sngset}
\{N_{j,\bf  k}={\lambda}_{j}-\mathbf{k}
\cdot{\tilde{\boldsymbol{\lambda}}}: N_{j,\bf  k}\in \mathbb{H}_\theta,\mathbf k\in\NN^{n_1},j=i_1,...,i_{n_1}\}
\end{equation}
have distinct complex arguments.

 Let $d_{j,\bf  k}$ be the direction
of $N_{j,\bf  k}$, that is the set $\{z:\arg z=\arg(N_{j,\bf  k})\}$. We note that 
the opposite directions, $\overline d_{j,\bf  k}$ are Stokes rays,\footnote{They are sometimes called Stokes lines, and often, in older literature, antistokes lines.} rays along which the Borel transforms are singular.
 \index{Stokes ray}

It can be easily seen that the set of $\boldsymbol{\lambda}$ which satisfy
(\ref{eq:sngset}) has full measure; for detailed definitions and explanations see  \cite{Duke}.

\subsection{Proofs of Theorem \ref{Ad-hoco} and \ref{Ad-hoc1}}\label{Smr} 
In the sequel, to conveniently work in vector spaces, we extend those functions $f$ defined only on proper subintervals of $\RR$ to all of $\RR$ by setting $f(x)=0$ for $x$ outside the subintervals in question. If $f$ and $g$ coincide on some interval $I$, we show that $A(f)=A(g)$ on $I$. The results for restricted-domain functions follow straightforwardly from the $0$-extension ones.

  \begin{Lemma}\label{allantider}{\rm 
    There is a linear antiderivative on $T[\RR,\RR]$ with the further property $(A_T (f))'(x_{\RR})=f(x_{\RR})$ for all $x_{\RR}\in \RR$.
  }\end{Lemma}
  \begin{proof}
    Let $I\subset \RR$ be a bounded  interval and $I^*$ be the corresponding interval in $V$.  Also let $x\in I^*$. Then there is a unique $x_{\RR}\in I$ such that $x-x_{\RR}$ is infinitesimal. Moreover,   $x_{\RR}=\inf\{y\in I:y>x\}=\sup\{y\in I:y<x\}$. We define $A_T(f)(x)=f(x_{\RR})(x-x_{\RR})$ for all infinitesimal $x-x_{\RR}$. Then $A_T$ is linear, and $(A_T(f))'=f^*$. For $|x|>\infty$ define $f^*(x)=0$ and $A_T f(x)=0$.
  \end{proof}
Note that the following proofs of Lemma \ref{Lemgen} and Theorem \ref{Ad-hoco} do not require the axiom of choice.
\begin{Lemma}\label{Lemgen}{\rm 
  Let $W$ be a subspace of $T[\RR,\RR]$, $W_1$ be a subspace on which there is a linear antiderivative $A_1$ and suppose $P$ is a projector on $W_1$. Then there exists an antiderivative $A$ on $T[\RR,\RR]$ that extends $A_1$. 
}\end{Lemma}
\begin{proof}
  We simply write (uniquely) $f\in T([\RR,\RR])$ as $f=Pf+(1-P)f=f_1+f_2$. We then simply check that  $A(f):=A_1(f)+A_T(f)$ has the desired properties; see Lemma \ref{allantider}.
\end{proof}

\begin{Note}{\rm 
 If $W_1$ is finite-dimensional, the existence of $P$ does \emph{not} depend on the axiom of choice. 
}\end{Note}

\begin{proof}[Proof of Theorem \ref{Ad-hoco}]

This follows from Lemma \ref{Lemgen} and the concrete choice of $P$.   We take $P$ to be the Lagrange interpolation polynomial with nodes $x_i=i$ for $i=0,...,n$:
$$L(f)(x)= \sum_{j=0}^{n} f(j) \ell_j(x);\ \ \ell_j(x) := \prod_{\begin{smallmatrix}0\le m\le k\\ m\neq j\end{smallmatrix}} \frac{x-m}{j-m}$$
$L(P)=P$ for any $P$ of degree $<n+1$, by the uniqueness of such a polynomial passing through $n+1$ points (cf. \cite{Nummethods}).
\end{proof}
\begin{proof}[Proof of Theorem \ref{Ad-hoc1}]
This follows from Lemma \ref{Lemgen} with $W_1= \{f\in T[\RR,\RR]:\exists F\in T[\No] \text{ such that} \, F \,\text{is differentiable and} \;(F|_{\RR})'=f\}$ and the linear choice $A(f)=F(x)-F(0)$ for $f\in W_1$. The existence of such a projector uses standard linear algebra and the axiom of choice.
  
\end{proof}
\begin{Note}
 {\rm  One can add a finite set of other ``good'' functions with known antiderivatives, if they exist. The construction of the Lagrange interpolation polynomial amounts to solving systems of linear equations which, for distinct nodes, have nonzero Vandermonde  determinants. For other functions, the nodes have to be chosen carefully. We do not pursue this, as $e^x$ and $e^{2x}$, for example, count as distinct functions, and a finite number of these would not be a significant improvement.}
\end{Note}

\subsubsection{Reduction to the case when the domain is $(x_0,\infty)$}\label{Srestr}
  \begin{enumerate}
  \item  The  functions for which we obtain ``good'' positive results are analytic except at the endpoints (which we understand to include $\pm \infty$).
  \item Assuming $f:(a,b)\to \RR$ is analytic on $(a-\epsilon,b)$, we ask whether it can be extended to $(a,b)\in \No$. If $b<\infty$ is a point of analyticity, then it is known that an extension exists, using the power series centered at $b$ \cite{FO}. Otherwise, the change of variables $g(x)=f[(b-x)/(x-a+1)]$ reduces the question to extending $g$ to positive infinitesimal $x$, which by the further transformation $g(x)=h(1/x)$ maps the problem to extensions from some $x_0\in \RR\cup\{-\infty\}$ past the gap at $+\infty$. 
Similar reasoning applies to antiderivatives.
  \item For the just-said reason, without loss of generality, we will assume that the intervals of interest are of the form $(x_0,\infty), x_0\in \RR\cup\{-\infty\}$, that the singularity is at $+\infty$ and we  will seek extensions and integrals to positive infinite $x\in\No$.
  \end{enumerate}

\subsection{Details of the proof of Theorem \ref{TT18}}\label{Sn4.1}
  Let 
\begin{equation}
  \label{eq:defnu}
 \nu_{\epsilon}:= \nu= 2\|s\|_{\infty}\pi^{-\frac12}\epsilon^{-1}
\end{equation}
and consider the  mollification $ f_{s;\epsilon}(x):=\pi^{-\frac12}\nu \int_{-\infty}^{\infty} e^{-\nu^2(x-t)^2}l_s(t)dt$. By standard complex analysis,  $f_{s;\epsilon}$ is entire, and straightforward estimates show that $\sup_{z\in\CC}|e^{-\nu^2|z|^2}f_{s;\epsilon}(z)|<\infty$. Note that,  by construction,  $\sup_{(t,x)\in\RR^2}|l_s(t)-l_s(x)|\le 2\|s\|_{\infty}|t-x|$. Thus, \eqref{eq:defnu} implies
\begin{multline}\label{mestdif}
  |f_{s;\epsilon}(x)-l_s(x)|=\pi^{-\frac12}\nu\left|\int_{-\infty}^{\infty} e^{-\nu^2(x-t)^2}(l_s(t)-l_s(x))dt\right|\\\le 2\|s\|_{\infty}\pi^{-\frac12}\nu\left|\int_{-\infty}^{\infty} e^{-\nu^2v^2}|v|dv\right|= 2\|s\|_{\infty}\pi^{-\frac12}\nu^{-1}<\epsilon.
\end{multline}
Hence $|F_{s;\epsilon}(x)-L_s(x)|\le \int_0^x|f_{s;\epsilon}-l_s|\le \epsilon x$. Note that $\inf_j s_j\le L_s(x)\le \sup_j s_j$. Since  $\mathsf{E}$ is nonnegative, $\inf_j s_j-\epsilon \le \omega^{-1}(\mathsf{E}(F_{s;\epsilon}))(\omega)\le \sup_j+\epsilon$. Note also that   $|L_{Ts}(x)-L_{s}(x)|\le 2\|s\|$. This implies  $\pm(F_{Ts;\epsilon}-F_{s;\epsilon})=\pm F_{Ts;\epsilon}-L_{Ts}+L_{Ts}-L_s+L_s-F_{s;\epsilon}\le 2\|s\|+2\epsilon x$. Hence $\re(\,\omega^{-1}|\mathsf{E}(F_{Ts;\epsilon}-F_{s;\epsilon})(\omega)|)<2\epsilon$. Since $l_s$ is $\epsilon$-independent,  an $\epsilon/2$ argument and  \eqref{mestdif} imply  $\pm(F_{s;\epsilon}-F_{s;\epsilon'})\le 2|\epsilon'-\epsilon|x$. Then  $\varphi(s):=\lim_{\epsilon\to 0} \re(\,\omega^{-1}\mathsf{E}(F_{s;\epsilon}(\omega)))$ exists,  is linear, shift invariant, and lies between $\inf_j s_j$ and $\sup_j s_j$.

\section{Descriptive set-theoretic results}\label{SLogic}
In this section we prove a result of mathematical logic that is instrumental for establishing some of our more general negative results from \S\ref{S2.2O}  and \S\ref{S2.3O}. We apply it to obtain negative results on extensions and integrations of familiar analytic functions. The bridge between mathematical foundations, in this higher generality, and analysis is illustrated in the proof of Theorem \ref{nonintext}.  As was mentioned in \S\ref{S2.2O}, we show that without stringent conditions governing the behavior of functions at $\infty$, extension or integration functionals exist only in ``trivial cases''.  Indeed, based on growth conditions alone, in the class of analytic functions outside some ball bounded by some $w$ in $\CC$ they exist if and only if $W$ is polynomially bounded, in which case the functions are just polynomials. Moreover, there is no analog of $L^1$ in $\No$ in the sense that even rapidly decaying entire functions do not suffice to ensure the existence of a linear antiderivative to some $\theta>\infty$.

\subsection{Positivity sets}\label{SL1}
The most primitive form of our negative results concern what we call positivity 
sets in the Cantor space $\{-2,-1,0,1,2\}^{\NN}$.

\begin{Definition}{\rm \label{D1.1.} $X$ is the Cantor space  $\{-2,-1,0,1,2\}^{\NN}$. A positivity set in  $X$ is an  $S\subseteq X$ such that the following holds for all  $x,y,z\in X$.
		
		i. If  $x+y = z$ and  $x,y \in S$, then  $z \in S$.
		
		ii. If  $x+y = z$ and  $x,y  \notin  S$, then  $z  \notin  S$. 
		
		iii. Suppose  $x \in X$ is eventually zero. Then  $ x \in S$ if and only if  $x$ has a positive last nonzero term.} \end{Definition}
\begin{Theorem}\label{T1.1.}{\rm  
		i. $\text{NBG}^{-}+\text{DC}$ proves that there is no Borel measurable positivity set in  $X$. 
		
		ii. $\text{NBG}^{-}+\text{DC}$  proves that if there exists a positivity set in  $X$, then there is a set of reals that is not Baire measurable.  
		
		iii. $\text{NBG}^{-}+\text{DC}$  does not prove that there exists a positivity set in  $X$. 
		
		iv.  There is no set-theoretic description with parameters such that
NBG proves the following. There is an assignment of real numbers to
the parameters such that the description uniquely defines a positivity
set in $X$.
This also holds for extensions of NBG by standard large cardinal hypotheses.} \end{Theorem}
            \begin{proof}

For i and ii, we argue in $\text{NBG}^{-}+\text{DC}$, and fix a positivity set  $S \subseteq  X=\{-2,-1,0,1,2\}^{\NN}$. We will show that  $S' = S \cap  \{-1,0,1\}^{\NN}$ is not Baire measurable in the Cantor space  $\{-1,0,1\}^{\NN}$. Suppose  $S'$ is Baire measurable in  $\{-1,0,1\}^{\NN}$. Let  $V \subseteq  \{-1,0,1\}^{\NN}$ be open, where  $V \Delta  S'$ is meager.

\begin{Lemma}\label{1.2.} {\rm   $V \ne \emptyset $.}\end{Lemma}

\begin{proof} Suppose  $V = \emptyset $. Then  $S'$ is meager in  $\{-1,0,1\}^{\NN}$. 
	Hence any bicontinuous permutation of  $\{-1,0,1\}^{\NN}$ sends  $S'$ onto a meager set in  $\{-1,0,1\}^{\NN}$. The finite (even countable) intersection of comeager sets in  $\{-1,0,1\}^{\NN}$ is comeager in  $\{-1,0,1\}^{\NN}$, and by the Baire category theorem, comeager sets in  $\{-1,0,1\}^{\NN}$ are nonempty. We use  the six bicontinuous permutations of  $\{-1,0,1\}^{\NN}$ that act coordinatewise and are the identity at coordinates after the first coordinate, and the six bicontinuous permutations of  $\{-1,0,1\}^{\NN}$ that act coordinatewise and are the minus function at coordinates after the first coordinate.  There must be an element of $\{-1,0,1\}^{\NN}\setminus S'$  whose image under all twelve of
 these bicontinuous permutations lies in $\{-1,0,1\}^{\NN}\setminus S'$. In particular, let  $f_1,f_2,f_3,g_1,g_2,g_3 \in \{-1,0,1\}^{\NN}\setminus S'$, where  $f_1,f_2,f_3$ agree on $\NN\setminus \{0\}$,  $g_1,g_2,g_3 = -f_1$ on $\NN\setminus \{0\}$ and $f_1(0),f_2(0),f_3(0)$, $g_1(0),g_2(0),g_3(0)$ are  $-1,0,1,-1,0,1$, respectively. Then  $f_2+g_3 = (1,0,0,...) \in \{-1,0,1\}^{\NN}$,  and   so  $ f_2+g_3 = (1,0,0,...) \in \{-1,0,1\}^{\NN}\setminus S'$, and has a positive last nonzero term, contradicting iii
in the definition of positivity set.

 \end{proof}

\begin{Lemma}\label{1.3.} {\rm   $S' = S \cap  \{-1,0,1\}^{\NN}$ is not Baire measurable in the Cantor space  $\{-1,0,1\}^{\NN}$. }\end{Lemma}

\begin{proof} Since  $V \ne \emptyset $, let  $ \alpha  = (\alpha_0,...,\alpha_k) \in \{-1,0,1\}^{k+1}$ be such that every  $f \in \{-1,0,1\}^{\NN}$ extending  $\alpha $ lies in  $ V$. We work in the Cantor space  $T(\alpha) = \{f \in \{-1,0,1\}^{\NN}$:  $f$ extends  $\alpha \}$. Since  $V \Delta  S'$ is meager,  $(V \cap  T(\alpha)) \Delta  (S \cap  T(\alpha))$ is meager in  $T(\alpha)$. That is,  $T(\alpha) \Delta  (S \cap  T(\alpha))$ is meager in  $T(\alpha)$, and so $ S^* = S \cap  T(\alpha)$ is comeager in  $T(\alpha)$. Using coordinatewise bicontinuous bijections of $ T(\alpha)$ as before, let  $f_1,f_2,f_3,g_1,g_2,g_3 \in S^*$, where  $f_1,f_2,f_3$ agree on  $\{k+2,k+3,...\}, g_1,g_2,g_3 = -f_1$  on $\{k+2,k+3,...\}$, and  $f_1(k+1),f_2(k+1),f_3(k+1),g_1(k+1),g_2(k+1),g_3(k+1)$ are  $-1,0,1,-1,0,1$, respectively. Then  $f_1+g_1 = (2\alpha_0,...,2\alpha_k,-2,0,0,...) \in X$, and so  $f_1+g_1 = (2\alpha_0,...,2\alpha_k,-2,0,0,...) \in S$ and has a
negative last nonzero term, contradicting iii
in the definition of positivity set.
\end{proof}

We have thus shown that  $S \cap  \{-1,0,1\}^{\NN}$ is not Baire measurable in  $\{-1,0,1\}^{\NN}$, and hence not Borel measurable. It is immediate that  $S \subseteq  \{-2,-1,0,1,2\}^{\NN}$  is not Borel measurable, establishing i.

Now since  $\{-1,0,1\}^{\NN}$ and  $\{0,1\}^{\NN}$ are homeomorphic, there is a non Baire measurable set in  $\{0,1\}^{\NN}$. 
As in \cite{JFA}, let $T \subseteq \{0,1\}^{\NN}$ be the set that results from removing from $\{0,1\}^{\NN}$ the elements that are eventually constant. Then $T$ is homeomorphic to $\RR\setminus\QQ \subseteq \RR$. Also since we have removed only countably many points from $\{0,1\}^{\NN}$, there is a subset of $T$ that is not Baire measurable in $T$. Hence there is a subset of $\RR\setminus\QQ$ that is not Baire measurable in $\RR\setminus\QQ$. Hence there is a subset of $\RR$ that is not Baire measurable in $\RR$. This (well-known argument, already given in \cite{JFA}) establishes ii.

\begin{Lemma}\label{1.4.} {\rm  $\text{NBG}^{-}+\text{DC}$ + ``all sets of reals are Baire measurable'' is consistent.}\end{Lemma}

\begin{proof} Since NBG$^-$ + DC is conservative over ZFDC, it suffices to
prove this with ZFDC. See Solovay \cite{Solovay}  and Shelah \cite{Shelah}. Solovay relies
on the consistency of ZFC + ``there exists a strongly inaccessible
cardinal'', whereas Shelah later only relies on the consistency of ZFC.  \end{proof}

We have thus established iii using ii.

\begin{Lemma}\label{L1.5.} {\rm  NBG + ``every set of reals set-theoretically definable with
real parameters is Baire measurable'' is consistent. Here the statement
in quotes is formulated as a scheme. This holds even if NBG is
augmented with standard large cardinal hypotheses (assuming that they
are consistent).
}\end{Lemma}

\begin{proof} 
This is proved in Solovay \cite{Solovay} and Shelah  \cite{Shelah}  with NBG
replaced by ZFC. This is equivalent in light of the conservativity  of
NBG over ZFC.
 \end{proof}

We now establish iv. We use ZFC below in light of the conservativity
of NBG over ZFC.

Let $\varphi(x, v_1, ..., v_k)$ be a formula of ZFC, and assume that ZFC proves

1) $(\exists v_1,...,v_k\in\RR)((\exists! x)(\varphi(x, v_1, ..., v_k))\wedge(\exists x)(\varphi(x, v_1, ..., v_k)\wedge x$ is a
positivity set in $X$)).

By Lemma \ref{L1.5.},  let $M$ be a model of $T$ + ``every set of reals set
theoretically definable with real parameters is Baire measurable'',
where $T$ is either ZFC or an extension of ZFC with standard large
cardinal hypotheses. Since $M$  satisfies ZFC, 1 holds in $M$. Working in
$M$, let $x$ be a positively set $S\subseteq X$  that is set-theoretically definable with real parameters. From the proof of ii, we:

a. Extract a set that is not Baire measurable in $ \{-1,0,1\}^{\NN}$ 
from  $S$.

b. Extract a set that is not Baire measurable in  $\{0,1\}^{\NN}$ from 
a.

c. Extract a set of reals that is not Baire measurable from b.

Thus in $M$, we have a set of reals, definable with real parameters, that is not 
Baire measurable. This contradicts the choice of $M$.       
            \end{proof}

\subsection{Proof of Theorem \ref{nonintext}}\label{RVE}

In this subsection we prove Theorem \ref{nonintext}, using the positivity sets of
Theorem \ref{T1.1.}  to conclude that the space of real-valued entire functions
cannot be naturally extended to the surreals. This is based on a fixed
convenient power series, which we view as a particularly transparent
toy model in advance of our main negative results.

In the next subsection we show that, in order to naturally extend families of real-valued entire functions past $\infty$ into the infinite surreals, strong conditions governing the behavior of the functions at $\infty$ is required.

We begin by recalling  the definitions of extension functional and antidifferentiation 
functional  introduced in the Introduction.

Let  $T[\RR,\mathbb{R}]$ be the real    vector space of all $f:\RR \rightarrow \mathbb{R}$ and   $\mathcal{E}$ be defined by \eqref{eq:defE}.

	\z We now apply Theorem \ref{T1.1.} to establish the following negative results for the space  $\mathcal{E}$   making use of Definitions \ref{DefExtf} and \ref{DefAf} of $\theta$-extension and $\theta$-antidifferentiation functionals.

	 First note that if  $L_f$ is a  $\theta$- extension functional on  $\mathcal{E}$ then  $L_f$ maps real polynomials  $P$ to  $P(\theta)$. Also, if  $L_f$ is a  $\theta$-antidifferention functional on $\mathcal{E} $ then $L_f$ maps real polynomials $P$ to $P^{(-1)}(\theta)$, where $P^{(-1)}$ is the antiderivative of $P$ with constant term $0$, see Definition \ref{Dd2}.

(a)		We use the following map $\rho :X \rightarrow  \mathcal{E}$, where $X = \{-2,-1,0,1,2\}^{\NN}$ as in \S\ref{SL1}. For $\xi\in X$,  $\rho (\xi)$ is the real-valued entire function given by the infinite radius of convergence power series

		\begin{equation}
			\label{eq:sum}
			\sum_{n\ge 0}\xi_n\frac{x^n}{n!}.
		\end{equation}
		The  $n!$ above can be replaced by larger convenient expressions, resulting in slower rates of growth. Let  $L_f$ be a  $\theta$-extension or  $\theta$-antidifferentiation operator on $\mathcal{E}$, where  $\theta $ is infinite. 

(b) Let  
$$S = \{\xi \in X: L_f(\rho (\xi)) > 0\}.$$
 If  $\xi+\eta = \zeta$ from  $X$, then  $L_f(\rho (\xi)) + L_f(\rho (\eta)) = L_f(\rho (\zeta))$. Hence if, furthermore,  $\xi,\eta \in S$, then  $\zeta \in S$. 

Also if, furthermore,  $\xi,\eta  \notin  S$, then  $\zeta  \notin  S$. Also, if any  $\xi \in X$ has a positive last nonzero term then  $\rho (\xi)$ is a polynomial with a positive leading coefficient. Hence  $L_f(\rho (\xi))$ is  $P(\theta)$ or  $A_0P(\theta)$, where  $A_0P$ is the antiderivative of  $P$ with constant term  $0$. Hence in either case,  $L_f(\rho (\xi)) > 0$, and so  $f \in S$. Therefore  $S$ is a positivity set.

We now complete the proof of Theorem \ref{nonintext}. To obtain i, ii, we simply
cite Theorem \ref{T1.1.} ii, iii. To obtain iii, suppose we have such a set-theoretic
description in iii with provability in NBG. Then we obtain a
corresponding set-theoretic description in Theorem \ref{T1.1.} iv with
provability in NBG,
contradicting Theorem \ref{T1.1.} iv. As in Theorem \ref{T1.1.} iv, we can use NBG
augmented by standard large cardinal hypotheses.

\section{Proof of Theorem \ref{2.1T71}}
The forward direction (a) implies (b) is immediate. We now show that
(b) implies (a).

The essence of the proof is to work in NBG, starting with any weight $W$
and three descriptions as in (b), without regard to the NBG
provability condition in (b). We give an explicit construction, from
these three givens, of a set X and prove in NBG that if $\mathcal{E}_W$ does not
consist entirely of polynomials, then $X$ is a positivity set. This
explicit construction is given below.

We now show that (b) implies (a), given this explicit construction.
Let $W$ be as given. Assume (b). Since NBG$^-$ is a conservative extension
of ZF, it suffices to show that ZF proves that  $\mathcal{E}_W$  (see Definition \ref{GC}) consists entirely
of polynomials. By set-theoretic absoluteness, it suffices to show
that ZFC proves that $\mathcal{E}_W$ consists entirely of polynomials. Suppose
this is false, and let $M$ be a countable model of ZFC in which  $\mathcal{E}_W$   does
not consist entirely of polynomials. Let $M'$ be a forcing extension of
$M$ given by Shelah \cite{Shelah} in which all sets of reals definable with real
parameters have the property of Baire. Let $M^*$ be an extension of $M'$ 
satisfying NBG in which all sets of reals set-theoretically definable
with real parameters have the property of Baire, and in which  $\mathcal{E}_W$  does
not consist entirely of polynomials.

We now make the explicit construction given below in $M^*$. We obtain a
positivity set which, in $M$, is set-theoretically definable with real
parameters, and therefore has the property of Baire. This contradicts
Theorem \ref{T1.1.}.

On the analytic side, the first step consists of constructing a rich enough family of functions of a given rate of growth. This is achieved, as illustrated in the previous section, by taking a sample function $f$  with positive Taylor coefficients and the prescribed rate of growth, and generating from it  a set of functions, belonging to  the same growth rate class, obtained by multiplying, in all possible ways, the coefficients of the sample function with $-2,-1,0,1,2$. This is an analytic analog of the space $f^{\Join}$ of \cite{JFA}; in the analytic case $f^{\Join}$ was  generated via Weierstrass products whereas here the similar space $f^{\triangleright}$ is generated by manipulating power series as illustrated in the previous section.

\begin{Proposition}\label{constrphi}{\rm 
		For any weight $W$ such that $\mathcal{E}_W$ contains non-polynomial entire functions, there is a  $\varphi_{W}\in C^{\omega}(\CC)$ with strictly  positive Taylor coefficients at zero (which is thus {\em not} a polynomial), explicitly constructed from $W$,  such that $\varphi_W\in\mathcal{E}_W$, and a positivity set explicitly constructed out of $W$.
	}\end{Proposition}
	\begin{proof}
		Let $b_k=\inf_{x\ge 1}x^{-k}W(x)$.  We first show that $b_k>0$. To get a contradiction assume that 
$b_k=0$ for some $k$. Since $W>0$ this means that there exists a sequence $\{r_j\}_{j\in\NN}$ such that $\lim_{j\to\infty}r_j^kW(r_j)=0$. Let $f\in \mathcal{E}_W$. By Cauchy's formula applied on  circles  of radii $r_j$ it follows that  $f^{(n)}(0)=0$  $\forall n>k$, {\em i.e.,} $f$ is a polynomial. 

Let $a_k=b_k 2^{-k}$ and $\varphi(z)=\sum_{k=0}^{\infty} a_k z^k$. Since $b_k>0$, we have $|\varphi(z)|\le \varphi(|z|)$. Let $|z|=\rho$. Then, for all $\rho >0$ we have  $\varphi(\rho)\le \sum_{k=0}^\infty (W(\rho) \rho^{-k})2^{-k}\rho^k =W(\rho)$. In particular, \eqref{eq:grW} holds with $C=1$, and the series of $f$ converges absolutely (in the classical sense) for any $z$.

\begin{Definition}\label{Defstar}
{\rm   If $c\in \left \{-2, -1,0,1,2 \right \}^{\mathbb{N}}$ and $\varphi(z)=\sum_{k=0}^\infty c_k z^k$, then we define 
$$(c\star \varphi)(z)=\sum_{k=0}^\infty c_n a_n z^n.$$
We also define
\begin{equation}
  \label{eq:triangleright}
  \varphi^{\triangleright}=\{c\star \varphi:c\in  X\};\ \ X:=\left \{-2, -1,0,1,2 \right \}^{\mathbb{N}}.
\end{equation}
}\end{Definition}
\begin{Note}{\rm
 For any $c\in  \left \{-2, -1,0,1,2 \right \}^{\mathbb{N}}$,   $\sup_{z\in\CC} |c\star \varphi |(z)\le 2W$ and thus for any $c\in  \left \{-2, -1,0,1,2 \right \}^{\mathbb{N}}$, $c\star\varphi\in\mathcal{E}_W$.
}\end{Note}

\bigskip 

\z
 The next step is to create a positivity set in the sense of of Definition \ref{D1.1.} from $\varphi$. 
Let $\theta>\infty$.  Then, 
\begin{equation}
\label{eq:posset}
c^+:= \left\{ c\in X:c\star \varphi(\theta)>0\right\}
\end{equation}
is a positivity set.  Indeed, a polynomial with a positive real leading coefficient is manifestly positive at $\theta$, and the two linearity conditions are immediate.

The rest of the proof of part (a) of Theorem \ref{2.1T71} shadows the one in (b) of \S\ref{RVE}.
\end{proof}

\begin{Note}
  {\rm The conclusion is, essentially, that for extensions and antiderivatives to exist, $W$ should be such that only polynomials are allowed. }
\end{Note}

\section{Proof of Theorem \ref{expogr.} }

 This is similar to the other negative proofs:  Let $\tilde{W}(x)=x^{\ln x}$ (superpolynomial but sub-exponential),  let $\varphi_{\tilde{W}}$ be as in Proposition \ref{constrphi} and $\varphi_{\tilde{W}}^{\triangleright}$ be as in \eqref{eq:triangleright}. Let $A$ be an antidifferentiation functional on $e^{-W}\varphi_{\tilde{W}}^{\triangleright}$. Note that $e^{-W} x^n \in e^{-W}\varphi_{\tilde{W}}^{\triangleright}$. Straightforward verification shows that $\{\xi\in X:  Af(\theta)<0 \}$ is a positivity set on $X$ (note the signs). The proof continues in the same way a that of Theorem \ref{2.1T71}.

  \begin{Note}
{\rm Our negative results rule  out natural ways of going beyond infinity that do not require further details about the nature of the singularity at infinity. }    
  \end{Note} For part iii, once more we mimic the proof of  Theorems 18 and 19 in \cite{JFA}.

\section{Proof of Theorem \ref{Tmainpos}}\label{S4.3}
\subsection{The class $\mathcal{F}_a$. Existence of extensions}

\begin{Definition}[Extension  by analytic continuation in the finite domain]\label{Ddefan} {\rm  Let $f$ be real-analytic on $[a,\infty)\subset\RR^+$, meaning that for any $x_0\in (a,\infty)$ the series
\begin{equation}
  \label{eq:eqh}
f(x_0+h)= \sum_{k=0}^\infty \frac{f^{(k)}(x_0)}{k!}h^k=: \sum_{k=0}^\infty c_k h^k
\end{equation}
has a positive radius of convergence. 

\subsubsection{Extensions for finite $x\in\No$ } 

Using the fact that we can uniquely write $x=x_0+s$, where $x_0\in\RR$ and $|s|\ll 1$, we define the extension  $\mathsf{E}(f)$ on $(a,\infty)\subset\No$ by
\begin{equation}
  \label{eq:polyns}
 \mathsf{E} (f)(x_0+s)=\sum_{k=0}^\infty \frac{f^{(k)}(x_0)}{k!}s^k=:\sum_{k=0}^\infty c_k s^k.
\end{equation} 
}\end{Definition}
\subsubsection{Extensions for  $x>\infty$.}\label{inftyGap}
\begin{Note}\label{N51}{\rm Taking $z=x^{-1/k}$ in \eqref{PSeq}, the functions in  $\mathcal{F}_a$ have series of the form
  \begin{equation}
    \label{eq:eqz}
    P(1/z)+\sum_{k=1}^{\infty}c_j z^j
  \end{equation}
where 
$\sum_{k=0}^{\infty}c_j z^j$ has a positive radius of convergence and  $P$ is a polynomial. Thus the question of extensions on $\mathcal{F}_a$ reduces to extending analytic functions. Note that, of course, \eqref{eq:eqz} is also absolutely convergent in the sense of Conway. 
  }\end{Note}

\begin{Lemma}\label{LemAnC}
 {\rm   The extension operator $\mathsf{E}$ is linear.  Moreover,  $\mathsf{E}$  has a genetic definition.} 
\end{Lemma}
\begin{proof}
The linearity of $\mathsf{E}$ is immediate. The existence of an equivalent genetic definition is known and is discussed in detail in the work of Fornasiero \cite[Pages 72-74]{FO}, a sketch of which is contained in the follow definition. 
\begin{Definition}\label{Def.6}
  {\rm 
(i) In \cite{FO}, the definition of the extension is given by an expression of the form
$$\{S^L|S^R\}$$
where 
$S^L$, $S^R$ are left (resp. right) options in the list of Taylor polynomials
\begin{equation}
  \label{eq:form}
  p_n(x,x^{\bf o})=\sum_{j=0}^n\frac{f^{(j)}(x^{\bf o})}{j!}(x-x^{\bf o})^j
\end{equation}
where $x^{\bf o}$ are left or right options for $x$. We can assume without loss of generality that $f$ is not a polynomial, otherwise the extension is trivial. Let $N=\min\{j>n:f^{(N)}(x^{\bf o})\ne 0\}$. Then,  $ p_n \in S^L$ if $f^{(N)}(x^{\bf o})(x-x^{\bf o})^{N}>0$ and $p_n \in S^R$ otherwise.

(ii) We will use expressions of the form \eqref{eq:form} also in the following setting: $x$ is infinite, $(x-x^{\bf o})$ is finite and $f^{(j)}(x^{\bf o})/f^{(j+1)}(x^{\bf o})$ is infinitesimal.
}\end{Definition}
\end{proof}

\subsection{The class $\mathcal{F}_a$. Existence of antiderivatives}
\subsubsection{Antiderivatives for finite $x\in \No$}
Let $f$ be a real analytic function and $F$ be a real antiderivative of $f$. 
 Of course $F$, the antiderivative of $f$,  is also real analytic, and the problem of defining the integral of $f$  reduces to that of the extension of the real analytic function $F$, since we define the surreal integral of $f$ between finite  $a,x\in \No$ by 
\begin{equation}
  \label{eq:defint}   \int_a^x f(t)dt=\mathsf{E}(F)(x)-\mathsf{E}(F)(a).
\end{equation}
The right side of \eqref{eq:defint} is defined in the previous section, and it is straightforward to check that  \eqref{eq:defint} provides an integral, with finite endpoints, having   the properties specified in Proposition \ref{existint}, where the positivity of the integral holds for  finite surreal numbers.
\subsubsection{Antiderivatives  for $x>\infty$}
Here we use series \eqref{PSeq} from \S{3.4.1}. Assume  \eqref{PSeq} converges for $x>r$, $r_1\in(r,\infty)$ and write
 \begin{multline}
   \label{eq:throughgap}
   \int_{a}^x f(t)dt=\int_{a}^{r_1} f(t)dt+\left(c_{-k}\ln t+\sum_{j\ne -k}\frac{c_j}{1-j/k}x^{-j/k+1}\right)\Bigg |_{r_1}^x\\ =
 \int_a^{\infty} \left ( f(t) -\sum_{j\le -k}c_{j}t^{-j/k}\right) dt
                        +c_{-k}\ln x 
                        +\sum_{j=-M}^{-k} \frac{c_j}{1-j/k} x^{-j/k+1}
 \end{multline}
where the infinite sum is a usual convergent series at ${r_1}$ treated as a  normal form for $x>\infty$. For $a \in \RR^+$,  $\int_a^{\infty}$ is the usual integral from $a$ to $\infty$ of a real function.

If the series in $1/x$ converges starting at $a$, one may readily check that the above amounts to termwise integration of the series.  Using the same transformation, $x^{-1/k}=z$, as in the previous section, we can transform this definition into a genetic one, using the genetic definition of log.

\begin{Proposition}\label{Pgood}{\rm 
The integral thus defined has properties (a)--(g) specified in Proposition  \ref{existint};  moreover,  item (g)  holds for {\em any} finite $a,b\in\No$.
}\end{Proposition}

subsection{Extensions and antiderivatives of EB transseriable functions}
\begin{Definition}\label{D27}{\rm 
(i) Lemma \ref{LEBstr} shows that  EB transseriable functions are analytic for large enough $x$.  Then the definition of $\mathsf{E}(f)$ for finite $x\in \No$ follows the same procedure that we used  for $\mathcal{F}_a$ for finite $x\in \No$.

(ii) For  $x>\infty$ we simply define
		\begin{equation}\label{BorelS}	\mathsf E(T)(x)=  T^*(x):= \sum_{\bfk\ge \bfk_0,l\in\NN}c_{\bfk,l}x^{\boldsymbol{\beta}\cdot \bfk}e^{-\bfk \cdot \boldsymbol{\lambda} x}x^{-l}
		\end{equation}
where  $x$ is an infinite surreal number written in normal form.
}\end{Definition}

In Definition \ref{D27} above, the infinite sum is an absolutely convergent series of normal forms, and thus has a Limit; see \S{2}. Accordingly, for infinite surreal $x$, Definition \ref{D27} in effect defines $\mathsf E(f)(x)$ to be the Limit of the absolutely convergent series of normal forms that arises on the right side of equation \ref{BorelS} by writing $x$ in normal form.

\begin{Proposition}\label{Pseries}{\rm 
Formal level one log-free transseries, see \eqref{trans2},  with real coefficients correspond 1-1 to surreal functions defined for  $x>\infty$, and the structures are isomorphic as differential algebras.}  
\end{Proposition}
\begin{proof}[Proof of Proposition \ref{Pseries}]\label{PPseries}
An infinite series is a limit of  polynomials. Convergence in the sense of transseries implies absolute convergence in the sense of Conway, when the variable in the transseries is replaced by an infinite surreal number written in normal form. Of course, surreal polynomials with real coefficients have the same algebraic properties as their real counterparts.

For differentiation, we take the restriction of the surreal function to the surreals of length $<\alpha$ (see \S{2}), where $\alpha$ is some limit ordinal. We take $\alpha_1$ to be the least ordinal greater than $n\alpha$ for all $n$ and $\alpha_2$  to be the least ordinal greater than $n\alpha_1$ for all $n$. Consider the infinitesimal $\delta$ whose sign expansion consists of a single plus followed by a string of ${\alpha_2}$ minuses. It is trivial to show that the reexpansion of  $ \sum_{k=0}^\infty c_k (h+\delta)^k$, where we expand  $(h+\delta)^k$ by the binomial formula, is absolutely convergent in the sense of Conway, since the supports of the coefficients of $\delta^mS_m(h)$ are disjoint for different values of $m$. By support we understand as usual the set of ordinals for which $r_{\alpha}\ne0$. The rest is a straightforward calculation. 
\end{proof}
\begin{Proposition}\label{Pe1}{\rm 
(i) The extension operator $\mathsf{E}$,    is a differential field isomorphism between EB transseriable functions and their respective images.

(ii) The definition can be recast as a genetic one.
  }\end{Proposition}
\begin{proof}
(i) is straightforward from Propositions \ref{consist} and \ref{Pseries} and (ii) is proved in the subsequent sections.
\end{proof}
 \subsubsection{The integral}We postpone the definition of the integral  of EB transseriable functions to \S\ref{Sintttranss}. 
For the time being, assume $f$ is a  Borel summable function. We first write $f=c/x +g(x)$, where now $g=O(1/x^2)$, and  then we write $\int_a^x f =c\ln s|_{a}^x+\int_a^\infty g-\int_x^\infty g$. 
\begin{Lemma}\label{PBorel}{\rm 
   $\int_x^\infty g$   is also Borel summable. }
\end{Lemma}
\begin{proof}
 We have $\lap^{-1}g=pA(p)$ with $A$ analytic and satisfying the conditions in Definition \ref{DefBE} and $\lap^{-1}\int_x^\infty g=A(p)$, which likewise satisfies the conditions in Definition \ref{DefBE}.  
\end{proof}
\begin{Definition}\label{DDefint}
  {\rm We define
    \begin{equation}
      \label{eq:defint}
      \int_a^x f =c\ln s|_{a}^x+\int_a^\infty g- \mathsf{E}(\int_x^\infty g).
    \end{equation}
}\end{Definition}
\begin{Proposition}
  {\rm The integral  in Definition \ref{DDefint} has properties (a)--(g) in Proposition  \ref{existint}.}
\end{Proposition}
\begin{proof}
  This follows straightforwardly from Lemma \ref{PBorel} and the properties of $\mathsf{E}$. 
\end{proof}
\subsection{The genetic construction. Borel summable series}\label{Sec5.5}
Clearly, we only need to show $\mathsf{E}$ has a genetic definition. We  arrange that the series begin with $k=0$, that $\beta=0$, and we let $x>\nu_1$ where $\nu_1$ is both $>\nu$ (cf. \S\ref{Sec5.3}) and large enough so that the results in \cite{CK} apply. By Borel summability, there are  $C,\rho \in\RR^+$ such that $|c_k|\le B_k:=C k! \rho ^{-k}$. The genetic definitions of the exponential and log (for leaders) are standard \cite[chapter 10]{GO} and we will take those for granted. 
\begin{Definition}\label{Dtildee=e}
{\rm   Write (cf. Note \ref{SumLT})
\begin{equation}
  \label{eq:defomega}
  \tilde{\mathsf{E}}(f)(x)=\left\{\sum_{k\in\ell(\rho x)} \frac{c_k}{x^{k+1}}-Ce^{-\rho x}x^b, S^L \Bigg|\sum_{k\in\ell(\rho x)}\frac{c_k}{x^{k+1}}+Ce^{-\rho x}x^b, S^R \right\},
\end{equation}
where 
$S^L,S^R$ are as in Definition \ref{Def.6}, and $C$ is greater than the maximum of 1 and

$$\sup_{x\in\RR, x>1} x^b e^{\rho x}\left|f(x)-\hspace{-.7 em}\sum_{{k\in\ell(\rho x)}}\hspace{-.2 em} \frac{c_k}{x^{k+1}}\right|,$$
the latter of which is finite by \cite{CK}.

For general $\ell$, the sums in the definition of $\tilde{\mathsf{E}}$ are normal forms, but they coincide with the usual sums when $x<\infty$ since then $\ell(\rho x)$ is a finite set.

}\end{Definition}
 
\begin{Note}
{\rm 	
For $x<\infty$, the sets $S^L$ and $S^R$ in Equation \ref{eq:defomega} provide tighter bounds than the corresponding  sums. Indeed, for finite $x$,  it is easy to show that some elements of $S^L (S^R)$ are always greater (resp. less) than the left (resp. right) sum. The relation betwee the sums and $S^L, S^R$ is the opposite for  $x>\infty$, unless $\amalg(x)$ is very small. This is also easy to check. 
}
\end{Note}
\begin{Proposition}\label{tildee=e}
  $\mathsf{E}=\tilde{\mathsf{E}}$.
\end{Proposition}
\begin{proof}
 First, let ${x_0}>\infty$ be such that $\amalg({x_0})\gg e^{-a{x_0}}$ for all $a\in\RR^+$. In this case, there is an  $a\in\RR^+$ such that $|{x_0}-{x_0}^{\bf{o}}|>a$ and thus dist$(S^L,S^R)>e^{-a{x_0}}$ for all $a\in\RR^+$. Therefore,
\begin{equation}
  \label{eq:atomega}
  \left\{\sum_{\ell(\rho {x_0})} \frac{c_k}{ {x_0} ^{k+1}}-2Ce^{-\rho {x_0}}{x_0}^b, S^L \Bigg|\sum_{\ell(\rho {x_0})} \frac{c_k}{ {x_0} ^{k+1}}+2Ce^{-\rho {x_0}}{x_0}^b, S^R \right\},\\
 \end{equation}
 which, by Proposition \ref{surtrunk}, equals 
  \begin{equation}
  \left\{\sum_{k=0}^\infty \frac{c_k}{ {x_0} ^{k+1}}-2Ce^{-\rho {x_0}}{x_0}^b\Bigg|\sum_{k=0}^{\infty} \frac{c_k}{ {x_0} ^{k+1}}+2Ce^{-\rho {x_0}}{x_0}^b\right\}=\sum_{k=0}^\infty c_k {x_0} ^{-k},
\end{equation}
which in turn implies 
\begin{equation}
  \label{eq:extf}
  \tilde{\mathsf{E}}(f)({x_0})=\sum_{k=0}^{\infty} \frac{c_k}{{x_0}^{k+1}}=\mathsf{E}(f)({x_0}).
\end{equation}
Now, for  $x=x_0+\epsilon$, where $\amalg({x_0})\gg e^{-a{x_0}}$ for all $a\in\RR^+$ and  $\epsilon=\amalg(\epsilon)$,  we  use $S_L$ and $S_R$ as in Definition \ref{Def.6} to obtain
  \begin{equation}
    \label{eq:anext}
    \tilde{\mathsf{E}}(f)(x_0+\epsilon)=\tilde{\mathsf{E}}(f)(x_0)+\sum_{k=0}^\infty \tilde{\mathsf{E}}(f)^{(k)}f(x_0)\epsilon^k,
  \end{equation}
where 
$$\tilde{\mathsf{E}}(f)^{(k)}(x_0)$$
is simply the formal $k$th derivative  of the formal series $\sum_{k=0}^\infty \frac{c_k}{x^{k+1}}$ 
where we take $x=x_0$ and interpret the sum as a (clearly well-defined) normal form. We note that, exactly as in routine manipulations of formal series, combining \eqref{eq:extf} with \eqref{eq:anext} leads to
\begin{equation}
  \label{eq:finalans}
   \tilde{\mathsf{E}}(f)(x)=\sum_{k\in\NN} \frac{c_k}{x^{k+1}}=\mathsf{E}(f)(x) \ \text{for all} \; x>\infty.
\end{equation}
It remains to check that \eqref{eq:anext} coincides with the value that $S^L,S^R$ give. For this, one can mimic the steps of \cite[pages 72--74]{FO}, since the arguments employed there, while presented for convergent series do not rely on convergence. 
 \end{proof}

\subsection{The genetic definition of $\mathsf{E}$  and of integrals of EB transseriable functions}
\subsubsection{Extensions of  functions  that have  (EB)-summable transseries at $\infty$.} 
Here we rely on \'Ecalle analyzability and transseries. See \cite{Book} for more details and references.
\begin{Note}
  {\rm 
We limit our analysis to Gevrey-one EB summable transseries. In fact, the definition of EB-summable transseries allows for any mixture of Gevrey types and any  combination of power series and exponentials. We do not pursue this generalization here. One reason is that it would require to a momentous undertaking due to the many cases that need to be covered and the technical tools involved (\'Ecalle acceleration, \'Ecalle cohesive continuation and so on). Secondly, some functions arising in applications may exhibit oscillatory behavior, and  a general $\CC-$valued theory of transseries is  not expected to exist in any generality; for instance oscillatory transseries do not, in general, have reciprocals. But in limited contexts such as those in \cite{IMRN, Duke}, they {\em do} exist and are well behaved under all operations except division. }
\end{Note}

\subsubsection{Extensions of  functions  that have EB-summable level one  transseries at $\infty$.} 
\begin{Proposition}\label{Pconvt}{\rm 
   \eqref{trans2} converges in the formal multiseries topology (see Definition \ref{defconv})  and is also well-defined as a normal form (with obvious generalizations for sur-complex $\lambda_j$). } \end{Proposition} 
\begin{proof}
 For transseries, this is proved in \cite{Book} and also in \cite{Edgar}. For normal forms, in the real-valued case this follows from Proposition \ref{sur7} by taking   $h_1=x^{-1},h_2=x^{\beta_1} e^{-\lambda_1 x}$ $,...,h_{m+1}=x^{\beta_m} e^{-\lambda_m x} $. The surcomplex extension is straightforward.

\end{proof}
Allowing for complex valued transseries, one can analyze all classical functions in analysis, general solutions of linear or nonlinear ODEs which are regular at $\infty$, or have a regular singular point at $\infty$, or can be normalized in such a way that they have a rank one singularity at $\infty$ \cite{Duke}. 

It is shown in \cite{Duke} that there is a one-to-one correspondence between solutions of \eqref{eqor} that go to zero as $x\to\infty$ and EB summable transseries of the form
\begin{equation}
\label{eq:gentr}
\mathbf y=  \sum_{\bfk\in \NN^n }\mathbf{c}_{\bfk}x^{\bfk\cdot\boldsymbol{\beta}+1}e^{-\bfk\cdot\bflam x}\mathbf{y}_{\bfk}(x),
\end{equation}
where the   $\bfy_{\bfk}$ are either zero, or of the form $\mathcal{L}Y_{\bfk}(p)$,
where $\mathcal{L}$ is the Catalan averaged Laplace transform; see the proof of Proposition \ref{consist} as well as \S\ref{3.4.1.} for the conventions governing $\boldsymbol{\beta \;\text{and} \;\lambda}$.
 
 Also,  in accordance  with Definition \ref{DEBstr}, it is shown in \cite{IMRN,Duke} that
\begin{equation}
\label{eq:bnds}
|c_{\bfk}|\sup_{x>x_0}|y_{\bfk}(x)|<\mu^{|\bfk|}
\end{equation}
for some $\mu>0$.

Also,
\begin{equation}\label{eq:fory0}
\bfy_0=O(x^{-M})
\end{equation}
where $M\in\NN$ can be chosen, via elementary transformations, to be as large as needed in the cases of interest, \cite{IMRN,Duke}.
\begin{Note}
	\rm{ 
		For simplicity, we assume that all coefficients and functions $\bfy$ are real-valued. This condition can be easily {\bf eliminated} considering surcomplex numbers, and writing inequalities separately for the real and imaginary parts of the functions involved. 
	}
\end{Note}
\subsubsection{Further assumption} The assumption adpoted here (and in \cite{IMRN}, where, perhaps confusingly, however, $\beta$ is denoted $-\beta$) is that $\Re\beta\in [-1,0)$. This holds for Painlev\'e equations and can be always arranged for in all linear ODEs, though it {\em cannot be simply made to hold in general nonlinear ODEs}. In \cite{Duke} this assumption was dropped, but to work in that generality and make the link to the theory of Catalan averages, one needs \'Ecalle pseudo-decelerations, which would lead to very cumbersome, albeit straightforward, calculations and we will not treat this case here. 

BE transseriable  functions are real-analytic by Lemma \ref{LEBstr}, so the definition of $\mathsf{E}$ for finite $x\in \No$ mimics that of $\mathcal{F}_a$. 

The extension beyond $\infty$ of \eqref{eq:gentr} is an immediate generalization of \eqref{eq:atomega}. Namely, we note that the function series in \eqref{eq:gentr} is {\em classically convergent}. To each of the component functions $\mathbf y_{\bf k}$ \cite{CK} applies --in fact uniformly in $\bfk$. Equation \eqref{eq:bnds} implies that for large $x_0$, all $x>x_0$ and all $N$, there is a $C_N$ such that
\begin{equation}
\label{eq:bds2}
\left| T-  \sum_{|\bfk|\in \{0,...,N\} }\mathbf{c}_{\bfk}x^{\bfk\cdot\boldsymbol{\beta}}e^{-\bfk\cdot\bflam x}\mathbf{y}_{\bfk}(x)\right|\le \max_{ |\bfk|\ge N} C_N x^{\bfk\cdot\boldsymbol{\beta}} e^{-\bfk\cdot\bflam x} =:E_N(x).
\end{equation}
Since the $\mathbf{y}_{\mathbf{k}}$ are EB-summable series, their genetic definition is provided by Definition \ref{Dtildee=e} and Proposition \ref{tildee=e}. This being the case, we may suppose the $\mathbf{y}_{\mathbf{k}}$ have already been defined, in which case we may write

\begin{Definition}\label{GendT}{\rm 
  \begin{equation}
    \label{eq:defte}
    \tilde{\mathsf{E}}(T)=\left\{\sum_{|\bfk|\le N} \mathbf{c}_{\bfk}x^{\bfk\cdot\boldsymbol{\beta}}e^{-\bfk\cdot\bflam x}\mathbf{y}_{\bfk}-2|E_N|, S^L\Bigg|\sum_{|\bfk|\le N }\mathbf{c}_{\bfk}x^{\bfk\cdot\boldsymbol{\beta}}e^{-\bfk\cdot\bflam x}\mathbf{y}_{\bfk}+2|E_N|,S^R\right\}
  \end{equation}
where $S^L, S^R$ are defined as in \S\ref{Sec5.5} above.
}\end{Definition}
\begin{Proposition}
  $\mathsf{E}=\tilde{\mathsf{E}}$.
\end{Proposition}
\begin{proof}
  The proof mimics the one in \S\ref{Sec5.5}.  Writing $x_0=\re(x_0)+ (x_0-\re (x_0))$ and noting that $e^{-\bfk\cdot\bflam \re(x_0)}$ is a real constant, we can absorb $e^{-\bfk\cdot\bflam x_0}$ into the $c_{\bfk}$ and assume $\re(x_0)=0$.  Now take the first $x_0$ such that $\amalg(x_0)=0$. Then, using the fact that $e^{-x}\ll x$ for all infinite $x$, we see that all the leaders in the normal form of $e^{-\bfk\cdot\bflam x}x^{\bfk\cdot\boldsymbol{\beta}}\mathbf{y}_{\bfk}$ are much smaller than any leader in $e^{-\bfk'\cdot\bflam x}x^{\bfk'\cdot\boldsymbol{\beta}}\mathbf{y}_{\bfk'} \ll x$ if $\bfk'\cdot\bflam>\bfk\cdot\bflam$. The rest of the proof, which concerns the case where $\amalg(x_0)=0$, is the same as in \S\ref{Sec5.5}. If $x=x_0+\epsilon$ with $\amalg(x_0)=0$ (and $\re(x)=\re(x_0)=0$ as above), we write $e^{-\bfk\cdot\bflam x}=e^{-\bfk\cdot\bflam x_0}[1+\bfk\cdot\bflam\epsilon+(\bfk\cdot\bflam\epsilon)^2/2+\cdots]$ and note that the power series can be absorbed in that of $\mathbf{y}_{\bfk}(x_0+\epsilon)$ and once more the proof continues exactly as in \S\ref{Sec5.5}.
\end{proof}

\subsubsection{Integration of surreal transseriable functions}\label{Sintttranss} 
\begin{Proposition}\label{termwise-int}
{\rm A transseries of the form (\ref{eq:gentr}), under the assumptions given there, can be integrated term-by term, resulting in an \'Ecalle-Borel summable transseries again of form (\ref{eq:gentr}). 
        }
\end{Proposition}
\begin{Note}\rm{ 
  The interpretation of a termwise integral on transseriable functions in the usual domain is a generalized integral from  $\infty$, which is the ``common point''. For the surreals, the integral is the {\em extension} of this usual integral and of the Hadamard finite part at $\infty$. 
}\end{Note}
\begin{proof}
  Integrating termwise a uniformly and absolutely convergent (in the classical sense)
    function series  is justified by
  the dominated convergence theorem. For a vector $\bfy$, integration
  is carried out component-wise so we can assume with no loss of generality
  that we are dealing with scalars.  For $y_0$, we simply use  definition \ref{DDefint}. 

  For higher indices $\bfk$, we need to solve
      ODEs of the type
      \begin{equation}\label{eq:higherk}
        f'=x^be^{-ax}y(x).
      \end{equation}
     After
       substituting $f=x^be^{-ax} g$, we obtain
        \begin{equation}\label{eqy1B}
          (a+p)G'=(b-1)G-Y' \Rightarrow G=-(a+p)^{-b}Y-b Y*(a+p)^{-b-1}.
        \end{equation}          
        By \cite{Menous}, the above operations of multiplication by bounded analytic functions and convolution commute with the Catalan averages, and therefore the integral of an EB-transseriable function $T$ is again an EB-transseriable function $T_1$, and we define  $\int T=\mathsf{E}(T_1)$.  Using Proposition \ref{Pe1}, the rest is straightforward. 
\end{proof}
\subsubsection{Classical functions covered} The following class of functions is amenable to extensions to $x\in\mathbf{No}$ where $x>\infty$:  Solutions of generic linear or nonlinear systems of ODEs (with an irregular singularity of rank 1 at $+\infty$) which, after normalization, satisfy  the conditions in \cite{Duke}.  These solutions may, of course, be complex-valued. As special cases we have the classical special functions in analysis: Airy Ai and Bi,  Bessel $I_\alpha$ $J_{\alpha}$ $Y_{\alpha}$ and $K_\alpha$,  Erf,   Ei, Painlev\'e and so on, possibly after changes of variables of the form $f\mapsto x^a f(x^b)$.  
Defining an integral amounts to solving a differential equation $y'=f$. If $f$ is a simple function, the equation reduces to the equations studied in \cite{Duke}.

\subsection{Some examples}\label{SSomexam}
\subsubsection{Ei}
Let  $C\in\RR$ be given by 
$$ C:=\sup_{x\in\RR, x>1} x^{1/2}\left|\mathrm{Ei}(x)-e^x\hspace{-.7 em}\sum_{k\in \ell(x)}\hspace{-.6 em} \frac{k!}{x^{k+1}}\right|,$$
   where $\ell(x)$ is the least term summation of the series defined in Equation \ref{eq:defell}. By \cite{CK}, the sup is finite; in fact, $C=3.54\cdots$. Then, for a surreal $x>1,$ where $\amalg(x)=0$ our definition of Ei is: 
\begin{equation}\label{DefEi}
\mathrm{Ei}(x)=\left\{e^x\hspace{-.7 em}\sum_{k\in \ell(x)}\hspace{-.6 em} \frac{k!}{x^{k+1}}-Cx^{-1/2}, S^L \Bigg|e^x\hspace{-.7em}\sum_{k\in \ell(x)}\hspace{-.6 em}\frac{k!}{x^{k+1}}+Cx^{-1/2}, S^R\right\}.
\end{equation}
Of course, if  $S^{L}(\mathrm{Ei})(x)=\emptyset$ (resp. $S^{R}(\mathrm{Ei})(x)=\emptyset$), then $S^{L}(\mathrm{Ei})(x)\pm \epsilon=\emptyset$ (resp. $S^{R}(\mathrm{Ei})(x)\pm \epsilon=\emptyset$).
The sums in \eqref{DefEi}  are interpreted as normal forms, after possible reexpansion. Indeed, they have finitely many terms if $x\in [1,\infty)$ and otherwise they are the Limits of series of normal forms that are absolutely convergent in the sense of Conway; see Propositions \ref{sur5}--\ref{sur7}. 
Similarly, $$\int_1^x t^{-1}e^t dt=\text{Ei}(1,x)$$ is given by $$\mathrm{Ei}(1,x)= \left\{e^x\hspace{-.7 em}\sum_{k\in \ell(x)}\hspace{-.6 em} \frac{k!}{x^{k+1}}-\mathrm{Ei(1)}-Cx^{-1/2}, S^L \Bigg|e^x\hspace{-.7em}\sum_{k\in \ell(x)}\hspace{-.6 em}\frac{k!}{x^{k+1}}-\mathrm{Ei(1)}+Cx^{-1/2}, S^R\right\}, $$
which leads to 
$$\mathrm{Ei}(\omega)=e^{\omega}\sum_{k=0}^{\infty}\frac{k!}{\omega^{k+1}};\ \mathrm{Ei}(1,\omega)=e^{\omega}\sum_{k=0}^{\infty}\frac{k!}{\omega^{k+1}}-\mathrm{Ei}(1).$$

The just-said value of $\mathrm{Ei}(\omega)$ is (after the reexpansion of terms) a purely infinite Conway normal form with no constant added. Since Ei is an antiderivative, it is defined up to a constant to be determined from the initial conditions. Our calculation shows that for Ei this constant is $0$. 
The fact that Ei corresponds to a purely infinite Conway normal form with constant $0$ at $\infty$ was conjectured by Conway and Kruskal. 

\subsubsection{Erfi} To calculate 
\begin{equation}
\label{eq:erfi}
\int_0^x e^{s^2}ds=\frac{\sqrt{\pi}}{2}\mathrm{erfi}(x)
\end{equation}
we solve
$f'=e^{s^2}$, where $f(0)=0$. With the change of variables $f(s)=s\exp(s^2)g(s)$ with $s=\sqrt{t}$ we get 
\begin{equation}
\label{eq:eqgt}
g'+\(1+\frac{1}{2t}\)g=\frac{1}{2t},\text{ where }\ g(z)=1+o(z)\ \text{as} \ z\to 0,
\end{equation}
whose Borel transform is 
\begin{equation}
\label{eq:eqGp}
(p-1)G'+\frac12 G=0,\  
\end{equation}
where $G(0)=1/2$. But this implies that
\begin{equation}\label{eq:eqGp}
  \frac14\left(\int_0^{\infty-0i}+\int_0^{\infty+0i}\right)\frac{e^{-tp}}{\sqrt{1-p}}dp\\=\frac12\int_0^1\frac{e^{-tp}}{\sqrt{1-p}}dp,
\end{equation}
and hence that the value of $f$ for $x>\infty$ is
\begin{equation}
\label{eq:trerfi}
\int_0^x e^{s^2}ds= \frac{e^{x^2}}{x\sqrt{\pi}}\sum_{n=0}^\infty  \frac{(2n-1)!!}{(2x^2)^n}.
\end{equation}

 \subsubsection{The Gamma function}
Based on the recurrence of the factorial, it is shown in \cite[page 99]{Book} that $\ln\Gamma$ has a Borel summed representation given by,
\begin{equation}
  \label{eq:lngamma}
  \ln\Gamma(n)=n(\ln n -1)-\frac{1}{2}\ln n +\frac{1}{2}\ln(2\pi)+\int_0^{\infty}\frac{\displaystyle
  1-\frac{p}{2}-\Big(\frac{p}{2}+1\Big)e^{-p}}{p^2(e^{-p}-1)}e^{-np}dp,
\end{equation} where the integrand is analytic at zero, as is seen by power series reexpansion at zero. It follows from our results that
\begin{equation}
  \label{eq:eqgamma}
  \ln \Gamma(\omega) =\omega \ln (\omega) - \omega - \tfrac{1}{2} \ln \left (\frac{\omega}{2\pi} \right ) + \frac{1}{12\omega} - \frac{1}{360\omega^3} +\frac{1}{1260 \omega^5}+\cdots \qquad \qquad
\end{equation}
where the coefficient of $\omega^{-k-1}$ is $ \frac{B_k}{k(k-1)}$ and the $B_k$ are the Bernoulli numbers.
By reexpansion we arrive at
\begin{equation}
  \label{eq:Stirling}
 \Gamma(\omega)= \omega^{\omega - \frac{1}{2}} e^{-\omega} \sqrt{2\pi} \left( 1 + \frac{1}{12\omega} + \frac{1}{288\omega^2} - \frac{139}{51840 \omega^3} - \frac{571}{2488320 \omega^4}+\cdots
        \right),
\end{equation}
where there is a closed form expression for the coefficients that is more intricate than that which can be obtained from \eqref{eq:eqgamma}. This is just Stirling's formula. Our results show that (after reexpansion of the series as a normal form) this is valid for all $x>\infty$.
\begin{Note}
  {\rm Clearly, one can the employ the approach used in this paper to define infinite sums. For example, since in classical mathematics $\ln \Gamma$  is a sum of logs indexed by ordinary positive integers, what we have calculated above might reasonably be denoted by
$$\sum_{k=1}^{\omega} \ln k.$$}
\end{Note}
\section{Logical issues}\label{XX}
Most of the paper directly involves {\bf No}, which is a highly
set-theoretic object. Therefore we need to be far more
careful about what axiom systems our developments live in
than we would for most mathematical developments. The
reader who is not interested in logical issues can safely
skip this section.

The issue arises immediately as in almost all mathematical
papers, the underlying formal system can be taken to be
ZFC. In fact, for most papers, it can be very conveniently
taken to be Z (Zermelo set theory), which is ZFC without
the axiom of choice or replacement. Yet here
{\bf No} is taken to be a proper class, too big to be a set.

We delineate two different approaches, each with their
advantages and disadvantages.

1. Literal. Here we use absolutely no coding devices in
order to simulate classes set theoretically, or to simulate
classes of classes as classes, and so forth. We take all of
the objects ``as they come" without reworking them. This
issue arises most vividly with partial functions $f:T[\RR]\to T[\No]$. $T[\RR]$ is unproblematic as it is a set. $T[\No]$ however
consists of (sets and) proper classes, and so under the
literal approach, we already need classes of proper
classes. There is even a problem with individual such $f$,
which from the literal point of view, is perhaps even more
of a problem. This is because $f$ is a set of ordered pairs -
typically $(g,h) = \{\{g\},\{g,h\}\}$, where $g$ is a set and $h$ is a
proper class. Note that $(g,h)$ is a class of classes of
classes. It is clear that NBG is far from sufficient for
the Literal approach.

2. Classes. Here we insist that we work only with the usual
sets and classes underlying the usual formal system NBG. We
take NBG to include the global axiom of choice, as is
customary. We use NBG$^-$ for NBG without any axiom of choice,
even for sets. In this approach, again there is no issue
with regard to the set $T[\RR]$ and its subsets. Also, there is
no problem with the proper class {\bf No}. However, the ordered
field $(\No,<,+,\cdot)$ has a problem. If taken literally, it is an
ordered quadruple of proper classes, and thus at least as
complex as a class of classes of classes, depending on how
quadruples are handled. However, there is a well known
coding device that appropriately simulates any finite
sequence of classes as a class. More generally, there is a
well known coding device that appropriately simulates any
class valued function whose domain is a set, as a class.
This takes care of partial functions  $f:T[\RR]\to T[\No]$.
However, it does not literally handle $T[\No]$ itself, as it
is a class of classes. In class theory, we treat $T[\No]$ as a
virtual class of classes, just as we would treat $\No$ as a
virtual class in the theory of sets. Adhering to set theory
creates some awkwardness that we wish to avoid, because of
$T[\No]$ and partial functions  $f:T[\RR]\to T[\No]$.
For the Literal approach, we use ZCI, which is Zermelo set
theory with the axiom of choice together with ``there exists
a strongly inaccessible cardinal''.

Under the literal approach, for the purpose of this
paper, the sets are the elements of the cumulative
hierarchy $V(\theta)$ up to the first strongly inaccessible cardinal $\theta$. The classes are the elements of $V(\theta+1)$, the
classes of classes are the elements of $V(\theta+2)$, and so forth.
Each $V(\theta+n)$ can be proved to exist in ZCI, but not $V(\theta+\omega)$.

Under either approach, DC is a weak form of the axiom of choice for
sets, called dependent choice. It asserts that for every binary
relation $R$ and every set $x$, there is a sequence $x = x_1,x_2,...$ such
that for all $i\ge 1$, if there exists $y$ such that $R(x_i,y)$, then
$R(x_i,x_i+1)$. NBG$^-$ (or even Z) proves that DC is equivalent to the
Baire category theorem.

NBG, NBG$^-$, NBG$^-$ + DC are conservative extensions of ZFC, ZF, ZFDC = ZF
+ DC, respectively. This means that the former systems are extensions
of the latter systems, and the former systems prove the same sentences
of set theory as the latter systems, respectively.

\section{Acknowledgments} The first author was  partially supported by the  NSF grant  DMS  1108794 and is thankful to Simon Thomas for very useful remarks.

\end{document}